\documentclass[12pt]{amsart}
\usepackage[margin=1in]{geometry}
\usepackage{hyperref}

\usepackage{amsthm, amsfonts, amsmath, amssymb, amsxtra}
\numberwithin{equation}{section}
\usepackage{cleveref}
\usepackage{enumitem}
\usepackage{mathtools}

\usepackage{cite}

\theoremstyle{definition}
\newtheorem{theorem}{Theorem}[section]
\newtheorem{corollary}[theorem]{Corollary}
\newtheorem{lemma}[theorem]{Lemma}

\newtheorem*{corollary*}{Corollary} 
\newtheorem*{definition*}{Definition}

\newtheorem{definition}[theorem]{Definition}
\newtheorem*{remark*}{Remark}
\newtheorem{remark}[theorem]{Remark}
\newtheorem{proposition}[theorem]{Proposition}
\newtheorem*{proposition*}{Proposition}

\crefname{theorem}{Theorem}{Theorems}
\crefname{lemma}{Lemma}{Lemmas}
\crefname{proposition}{Proposition}{Propositions}
\crefname{definition}{Definition}{Definitions}

\newcommand{\norm}[1]{\ensuremath{\left\| {#1} \right\|}}
\DeclareMathOperator{\dv}{div}

\DeclareMathOperator{\dist}{dist}
\DeclareMathOperator{\diam}{diam}
\DeclareMathOperator{\osc}{osc}

\newcommand{\ee}{\ensuremath{\varepsilon}}
\newcommand{\bb}{\ensuremath{\mathbf{b}}}
\newcommand{\ff}{\ensuremath{\mathbf{F}}}
\newcommand{\bbb}{\ensuremath{\mathbf{B}}}

\title[Linearized Monge-Amp\`ere equations in divergence form]
{Interior Harnack inequality and H\"older estimates for linearized Monge-Amp\`ere Equations
in divergence form with Drift}
\author{Young Ho Kim}
\address{Department of Mathematics, Indiana University, Bloomington, IN 47405, USA}
\email{yk89@iu.edu}

\begin{document}

\begin{abstract} 
	In this paper, 
	we study interior estimates for solutions to linearized Monge-Amp\`ere equations
	in divergence form with drift terms and the right-hand side containing 
	the divergence of a bounded vector field.
	Equations of this type appear in the study of semigeostrophic equations in meteorology
	and the solvability of singular Abreu equations in the calculus of variations
	with a convexity constraint.
	We prove an interior Harnack inequality and H\"older estimates for solutions to 
	equations of this type in two dimensions,
	and under an integrability assumption on the Hessian matrix of the Monge-Amp\`ere potential in higher dimensions.
	Our results extend those of Le (\emph{Analysis of Monge-Amp\`ere equations},
	Graduate Studies in Mathematics, vol.240, American Mathematical Society, 2024)
	to equations with drift terms.
\end{abstract}

\maketitle

\section{Introduction and Statements of the Main Results}

In this paper, we are interested in the interior estimates for solutions $u:\Omega\rightarrow\mathbb{R}$ to 
linearized Monge-Amp\`ere equations of the form
\begin{equation}\label{eqn}
\begin{aligned}
	-\dv(\Phi Du + u\bbb) + \bb \cdot Du = f - \dv\ff
\end{aligned}
\end{equation}
in a bounded domain $\Omega\subset\mathbb{R}^n$, $n\geq 2$,
where $\bb$, $\bbb$, $\ff$: $\Omega\rightarrow\mathbb{R}^n$ are bounded vector fields,
$f\in L^n$, and
\begin{equation}\label{cofdef}
\begin{aligned}
	\Phi = (\Phi^{ij})_{1\leq i,j\leq n} = (\det D^2 \varphi) (D^2\varphi)^{-1}
\end{aligned}
\end{equation}
is the cofactor matrix of the 
Hessian matrix 
$$D^2\varphi = (D_{ij}\varphi)_{1\leq i,j\leq n} 
= \left(\frac{\partial^2\varphi}{\partial x_i \partial x_j}\right)_{1\leq i,j\leq n}\text{.}$$
Here $\varphi$ is a $C^3$ convex Monge-Amp\`ere potential satisfying
\begin{equation}\label{maeqn}
\begin{aligned}
	0 < \lambda \leq \det D^2\varphi \leq \Lambda
	\quad\text{ in } \Omega
	\text{.}
\end{aligned}
\end{equation}

As the cofactor matrix $\Phi$ is divergence-free, that is,
$D_i \Phi^{ij} = 0$ for all $j$, 
the left-hand side of (\ref{eqn}) can also be written in 
nondivergence form and we have
\begin{align*}
-\Phi^{ij}D_{ij}u + (\bb-\bbb)\cdot Du -(\dv\bbb) u = f-\dv\ff
\text{.}
\end{align*}
We will focus on the divergence form and the case when $\ff\neq 0$,
and obtain interior estimates for $u$ using its integral information.

\subsection{Linearized Monge-Amp\`ere Equations}
Linearized Monge-Amp\`ere equations arise in several contexts
such as affine maximal surface equation in affine geometry \cite{TW00, TW05, TW08},
K\"ahler metrics of constant scalar curvature in complex geometry \cite{CHLS14, CLS, Don05, FS},
solvability of Abreu type equations in complex geometry 
and in the calculus of variations with a convexity constraint \cite{Ab, Zhou, CW, CR, Singular_Abreu, Convex_Approx, Abreu_HD, KLWZ, WZ},
and semigeostrophic equations in meteorology \cite{ACDF, Fig, Loeper05, SG_LMA}.

For a strictly convex function $\varphi\in C^2(\Omega)$ satisfying (\ref{maeqn}),
the cofactor matrix $\Phi$ is positive definite,
but we cannot expect structural bounds on its eigenvalues.
Hence, the linearized Monge-Amp\`ere operator is an elliptic operator
that can be degenerate and singular.

Starting with the seminal result of Caffarelli-Guti\'errez \cite{CG97}
on the homogeneous equation 
\begin{align*}
\Phi^{ij} D_{ij} u = \dv (\Phi Du) = 0
\text{,}
\end{align*}
linearized Monge-Amp\`ere equations have been studied by many authors
\cite{GN1, GN2, Boundary_Harnack, Twisted_Harnack, LNW2p, LN, LS, S, Maldonado, MalW2, M17, TZ}.
The term $\dv \ff$ in (\ref{eqn}) appears in the study of semigeostrophic equations in meteorology.
Specifically, we have equations of the form 
\begin{equation}\label{eqndiv}
\begin{aligned}
	\dv (\Phi Du) = \dv\ff
	\text{.}
\end{aligned}
\end{equation}
See \cite[equation (13)]{Loeper05},
\cite[equation (1.5)]{SG_LMA} and \cite[equation (15.51)]{Le24}.

For equations of this type, 
Loeper \cite{Loeper05} proved the H\"older estimate of solutions using integral information of $u$ 
under the assumption that $\det D^2\varphi$ is close to a constant.
Roughly speaking, Loeper needed this condition to apply the results of 
Murty-Stampacchia \cite{MS} and Trudinger \cite{Trudinger1973};
see Section 1.3 for more information.
Le \cite{SG_LMA} proved the same result when $n=2$ with just the assumption in (\ref{maeqn}).
Le \cite[Theorem 15.6]{Le24} also proved the H\"older estimate when $n\geq 3$
under an integrability assumption on the Hessian matrix $D^2\varphi$,
that is, $D^2\varphi\in L^s$ for $s>n(n-1)/2$.
This equation was also studied by Wang \cite{W},
where the H\"older estimate is proved 
under an integrability assumption on $(D^2\varphi) ^{1/2}\ff$;
more precisely, when $(D^2\varphi) ^{1/2}\ff\in L^q$, $q>n$.
In Wang \cite{W}, the upper bound for the H\"older norm contains the $L^\infty$ norm of the solution $u$,
while in Le \cite{Le24}, the $L^p$ norm ($p>1$) is used.

The main difference between (\ref{eqn}) and (\ref{eqndiv})
is the existence of drift terms $-\dv(u\bbb)$ and $\bb\cdot Du$.
When $\ff=0$, equations of the form (\ref{eqn}) 
with nonzero drift terms ($\bb, \bbb\neq 0$) have been studied
by Maldonado \cite{Maldonado, M18, M19} and Le \cite{Le18, Twisted_Harnack}.
These appear in the solvability of singular Abreu equations in higher dimensions in complex geometry 
and in the calculus of variations with a convexity constraint 
\cite[equations (2.2) and (2.5)]{KLWZ}.

\subsection{The Main Results}
In this paper, we will consider equations 
of type (\ref{eqn}) that have both the drift terms, and also $\dv \ff$, 
in dimension two and
under an integrability assumption on $D^2\varphi$ in higher dimensions.
Our main results are the following theorems on interior Harnack inequality and H\"older estimates.
They extend the result of Le \cite{Le24} to equations with drift terms.

Our first result is the following Harnack inequality.
\begin{theorem}[Harnack inequality]\label{thm:harnack}
	Let $\varphi\in C^3(\Omega)$ be a convex function satisfying (\ref{maeqn}).
	Suppose that $\ff,\bb,\bbb\in  W_{\mathrm{loc}}^{1,n}(\Omega;\mathbb{R}^n)
	\cap L_{\mathrm{loc}}^\infty(\Omega;\mathbb{R}^n)$, 
	$f\in L_{\mathrm{loc}}^n(\Omega)$, $n/2<r\leq n$, and $\dv \bbb\leq 0$.
	Assume that $S_\varphi(x_0,2h)\Subset\Omega$, 
	where $S_\varphi(x_0,\cdot)$ is the section defined in Definition \ref{def:sec}.
	Let $u \in W^{2,n}(S_\varphi(x_0,h))$ 
	be a nonnegative solution to (\ref{eqn}) in $S_\varphi(x_0,h)$ and let $t\leq h/2$.
	Further assume that
	\begin{enumerate}
		\item either $n= 2$, or 
		\item $n\geq 3$ and $\ee^*(n,\lambda,\Lambda) + 1 > \frac{n}{2}$,
		where $\ee^*$ is the exponent in the interior $W^{2,1+\ee}$ estimate
		for the Monge-Amp\`ere equation in Theorem \ref{intw2pest}.
	\end{enumerate}
	Then, there are positive constants $C$ and $\gamma$ such that
	\begin{equation*}
	\begin{aligned}
		\sup_{S_\varphi(x_0,t)} u
		\leq C\left(
			(\norm{\ff}_{L^\infty(S_\varphi(x_0,h))}+\norm{f}_{L^r(S_\varphi(x_0,h))})t^\gamma 
			+\inf_{S_\varphi(x_0,t)} u
		\right)
		\text{.}
	\end{aligned}
	\end{equation*}
	Here the constants $C$ and $\gamma$ are given by
	\begin{equation*}
	\begin{aligned}
		\gamma &=\gamma(n,\lambda,\Lambda,r) > 0
		\text{,}\quad\text{and }\\
		C&=C(n,\lambda,\Lambda,r,\ee^*,\norm{\bb}_{L^\infty(S_\varphi(x_0,h))},
			\norm{\bbb}_{L^\infty(S_\varphi(x_0,h))}, \norm{\dv \bbb}_{L^n(S_\varphi(x_0,h))},
			h, \diam (S_\varphi (x_0,2h)))
		\text{.}
	\end{aligned}
	\end{equation*}
\end{theorem}
We will prove Theorem \ref{thm:harnack} in Section 4.

From the Harnack inequality, we have the following interior H\"older estimates.
\begin{corollary}[H\"older estimates with $L^\infty$ norms]\label{linftyhoelder}
    Let $\varphi\in C^3(\Omega)$ be a convex function satisfying (\ref{maeqn}).
	Assume that $\ff,\bbb,\bb \in L_{\mathrm{loc}}^\infty(\Omega;\mathbb{R}^n) 
	\cap W_{\mathrm{loc}}^{1,n}(\Omega;\mathbb{R}^n)$,
	$f\in L_\mathrm{loc}^n(\Omega)$,
	$\dv\bbb\leq0$, $n/2<r\leq n$,
	and $S_\varphi(x_0, 4h_0)\Subset\Omega$.
	Let $u\in W_{\mathrm{loc}}^{2,n}(S_\varphi(x_0, 4h_0))$
	be a solution to (\ref{eqn}) in $S_\varphi(x_0, 4h_0)$.
	Further assume that
	\begin{enumerate}
		\item either $n= 2$, or 
		\item $n\geq 3$ and $\ee^*(n,\lambda,\Lambda) + 1 > \frac{n}{2}$,
		where $\ee^*$ is the exponent in the interior $W^{2,1+\ee}$ estimate
		for the Monge-Amp\`ere equation in Theorem \ref{intw2pest}.
	\end{enumerate}
	Then, there are positive constants $C$ and $\gamma$ such that
	for all $x,y\in S_\varphi(x_0, h_0)$, we have
	\begin{equation}
	\begin{aligned}
		|u(x)-u(y)|\leq C
		\left(\norm{\ff}_{L^\infty(S_\varphi(x_0, 2h_0))}+\norm{f}_{L^r(S_\varphi(x_0, 2h_0))}
			+\norm{u}_{L^\infty(S_\varphi(x_0,h_0))} \right) |x-y|^\gamma
		\text{.}
	\end{aligned}
	\end{equation}
	Here $\gamma$ depends on $n$, $\lambda$, $\Lambda$, $\ee^*$,
	$\norm{(\bb,\bbb)}_{L^\infty(S_\varphi(x_0, 2h_0))}$,
	$\norm{\dv\bbb}_{L^n(S_\varphi(x_0, 2h_0))}$,
	$\diam(S_\varphi(x_0, 4h_0))$, and $h_0$,
	and $C$ depends on 
	$\norm{(\bb,\bbb)}_{L^\infty(S_\varphi(x_0, 2h_0))}$, 
	$\norm{\dv\bbb}_{L^n(S_\varphi(x_0, 2h_0))}$,
	$\mathrm{diam}(S_\varphi(x_0,4h_0))$,
	$n$, $\lambda$, $\Lambda$, $r$, $\ee^*$, 
	and $h_0$.
\end{corollary}
We will prove Corollary \ref{linftyhoelder} in Section 6.

With stronger assumptions on the integrability of the Hessian matrix $D^2\varphi$
in higher dimensions,
we can obtain the following interior H\"older estimate, 
where the $L^\infty$ norm of the solution $u$ in Corollary \ref{linftyhoelder}
is replaced by its $L^2$ norm.
\begin{theorem}[H\"older estimates with $L^2$ norm]\label{mainthm}
    Let $\varphi\in C^3(\Omega)$ be a convex function satisfying (\ref{maeqn}).
	Assume that $\ff,\bbb,\bb \in L_{\mathrm{loc}}^\infty(\Omega;\mathbb{R}^n) 
	\cap W_{\mathrm{loc}}^{1,n}(\Omega;\mathbb{R}^n)$,
	$f\in L_\mathrm{loc}^n(\Omega)$,
	$\dv\bbb\leq0$, $n/2<r\leq n$,
	and $S_\varphi(x_0, 4h_0)\Subset\Omega$.
	Let $u\in W_{\mathrm{loc}}^{2,n}(S_\varphi(x_0, 4h_0))$
	be a solution to (\ref{eqn}) in $S_\varphi(x_0, 4h_0)$.
	Further assume that
	\begin{enumerate}
		\item either $n= 2$, or 
		\item $n\geq 3$ and $\ee^*(n,\lambda,\Lambda) + 1 > \frac{n(n-1)}{2}$,
		where $\ee^*$ is the exponent in the interior $W^{2,1+\ee}$ estimate
		for the Monge-Amp\`ere equation in Theorem \ref{intw2pest}.
	\end{enumerate}
	Then, there are positive constants $C$ and $\gamma$,
	where $\gamma$ depends on $n$, $\lambda$, $\Lambda$, $\ee^*$,
	$\mathrm{diam}(S_\varphi(x_0,4h_0))$,
	$\norm{(\bb,\bbb)}_{L^\infty(S_\varphi(x_0, 2h_0))}$, 
	and $h_0$,
	and $C$ depends on $n$, $\lambda$, $\Lambda$, $r$, $\ee^*$, 
	$\norm{(\bb,\bbb)}_{L^\infty(S_\varphi(x_0, 2h_0))}$,  
	$h_0$, and
	$\mathrm{diam}(S_\varphi(x_0,4h_0))$, such that
	for all $x,y\in S_\varphi(x_0, h_0)$, we have
	\begin{equation}
	\begin{aligned}
		|u(x)-u(y)|\leq C
		\left(\norm{\ff}_{L^\infty(S_\varphi(x_0, 2h_0))}+\norm{f}_{L^r(S_\varphi(x_0, 2h_0))}
			+\norm{u}_{L^2(S_\varphi(x_0,2h_0))} \right) |x-y|^\gamma
		\text{.}
	\end{aligned}
	\end{equation}
\end{theorem}
We will prove Theorem \ref{mainthm} in Section 6.

\begin{remark}
In Theorem \ref{mainthm},
we use the $L^2$ norm of the solution $u$ in the estimate
(in fact, any $L^p$ norm for $p>0$ can be used);
in Corollary \ref{linftyhoelder}, the $L^\infty$ norm of $u$ is used in the estimate.
The improvement in Theorem \ref{mainthm} comes at the cost of 
having to assume stronger integrability of $D^2\varphi$ when $n\geq 3$, namely, 
$1+\ee^* > \frac{n(n-1)}{2}$.
This is because we need this condition in the proof of the interior estimate in Lemma \ref{lemma4}.
It would be interesting to see if the condition $1+\ee^*>\frac{n(n-1)}{2}$ 
can be relaxed in Theorem \ref{mainthm}.
\end{remark}

\begin{remark} 
Note that, by Caffarelli \cite{Cafw2p} (also see \cite[Theorem 6.13]{Le24}),
for any $p>1$ and any convex function $\varphi$ satisfying (1.3), 
we have $D^2\varphi \in L_{\mathrm{loc}}^p(\Omega)$, 
provided that $\Lambda/\lambda-1\leq e^{-C(n)p}$ for some large constant $C(n)>1$.
\end{remark}

\begin{remark}
In our theorems,
we require $\varphi$ to be $C^3$ in the domain.
However, our estimates do not depend on the regularity of $\varphi$
but only on the constants $\lambda$, $\Lambda$, and $n$.
The functions $\ff$, $\bb$, $\bbb$ are assumed to be in 
$L_{\mathrm{loc}}^\infty(\Omega;\mathbb{R}^n) 
\cap W_{\mathrm{loc}}^{1,n}(\Omega;\mathbb{R}^n)$ 
and $f$ to be in $L_\mathrm{loc}^n(\Omega)$,
but the estimates depend only on the quantities stated.
\end{remark}

\subsection{Related Results for Equations in Divergence Form}
Divergence form equations
\begin{equation}\label{geneqn}
\begin{aligned}
	-\dv (a Du + u\bbb) + \bb \cdot Du + cu = f - \dv\ff
	\text{\quad in }\Omega\subset \mathbb{R}^n
\end{aligned}
\end{equation}
have been studied in the case when the symmetric coefficient matrix $a=a(x)$ is not uniformly elliptic,
but instead satisfies
\begin{equation*}
\begin{aligned}
	\rho (x) I_n \leq a(x) \leq \mu (x) I_n
\end{aligned}
\end{equation*}
for nonnegative functions $\rho$ and $\mu$, where $I_n$ is the $n\times n$ identity matrix.
Murty-Stampacchia \cite{MS} and Trudinger \cite{Trudinger1973} 
proved $L^\infty$ and H\"older estimates for solutions to equations of the form (\ref{geneqn}) 
with integrability assumptions on $\mu$ and $\rho^{-1}$.
Specifically, it is assumed that $\mu\in L^p$ and $\rho^{-1}\in L^q$, 
where $\frac{1}{p} + \frac{1}{q}<\frac{2}{n}$.
These extend the classical results of De Giorgi \cite{DeGiorgi}, Nash \cite{Nash}, and Moser \cite{Moser} 
for uniformly elliptic equations,
when $\rho$ and $\mu$ are positive constants.

Bella-Sch\"affner \cite{B-S} extended the above results in the case of equations of the form
\begin{equation*}
\begin{aligned}
	-\dv (a Du) = 0
\end{aligned}
\end{equation*}
in $\Omega\subset\mathbb{R}^n$, under the assumption that $\frac{1}{p} + \frac{1}{q}<\frac{2}{n-1}$.
This result is essentially sharp,
as Franchi-Serapioni-Serra Cassano \cite[Theorem 2]{FSS} proved that a 
counterexample exists if $n\geq 4$ and $\frac{1}{p}+\frac{1}{q}>\frac{2}{n-1}$.

In the case when the matrix $a=\Phi$ is the cofactor matrix of the Hessian matrix
$D^2\varphi$, where $\varphi$ satisfies (\ref{maeqn}), we have
\begin{equation*}
\begin{aligned}
	a= (\det D^2\varphi) (D^2\varphi)^{-1}\geq 
	\frac{\det D^2\varphi}{\norm{D^2\varphi}} I_n
	\text{.}
\end{aligned}
\end{equation*}
As $\det D^2\varphi\geq \lambda$ and $D^2\varphi\in L^{1+\ee^*}$
by the $W^{2,1+\ee}$ estimate for Monge-Amp\`ere equations (see Theorem \ref{intw2pest}),
$\rho^{-1}\in L^{1+\ee^*}$.
Furthermore, $\rho^{n-1}\mu$ is bounded by (\ref{maeqn}), and thus $\mu\in L^{(1+\ee^*)/(n-1)}$.
Therefore, we get
\begin{equation*}
\begin{aligned}
	\frac{1}{p}+\frac{1}{q} = \frac{1}{1+\ee^*}+\frac{n-1}{1+\ee^*}
	=\frac{n}{1+\ee^*}
	\text{.}
\end{aligned}
\end{equation*}
Note that, with only the assumption that $\ee^*>0$,
this is smaller than $\frac{2}{n-1}$ only when $n=2$.
When $n\geq 3$, the assumption $1+\ee^*> n(n-1)/2$ in Theorem \ref{mainthm}
and in Le \cite{Le24} corresponds to $\frac{n}{1+\ee^*}<\frac{2}{n-1}$.
Compared to the results of Bella-Sch\"affner,
these cover the equations with nonzero right-hand side (especially the case when $\ff\neq 0$),
with the assumption that the matrix $a=(\det D^2\varphi)(D^2\varphi)^{-1}$. 

\subsection{Methods of the proofs}
We briefly discuss the differences in the proofs of the results in this paper,
the results of Le \cite{SG_LMA, Le24}, and the results of Wang \cite{W}.

The proof of interior H\"older estimates in Le \cite{SG_LMA, Le24}
used the fine properties of the Green's function for the linearized
Monge-Amp\`ere operator \cite{LMA_Green, Boundary_Harnack}.
Other tools used in the proof are
De Philippis-Figalli-Savin and Schmidt's $W^{2,1+\ee}$ estimate \cite{DPFS, Schmidt}
in the case $n=2$, and the Monge-Amp\`ere Sobolev inequality.
The $W^{2,1+\ee}$ estimate is replaced by an integrability assumption
for $D^2\varphi$ when $n\geq 3$.
The results for the Green's function for the linearized Monge-Amp\`ere
operator with drift terms are not available, 
so we take an alternative approach in our proofs.

Wang \cite{W} uses the De Giorgi iteration technique,
in addition to the Monge-Amp\`ere Sobolev inequality,
in the proof of interior H\"older estimates.
We will use the Moser iteration techniques similar to the ones in 
Gilbarg-Trudinger \cite[Chapter 8]{GT} and Trudinger \cite{Trudinger1973},
and the Monge-Amp\`ere Sobolev inequality in our proofs.

The rest of this paper is organized as follows.
In Section 2, we present definitions and prior results used in the proofs of the results.
In Section 3, we establish global $L^\infty$ estimates for solutions to (\ref{eqn}).
In Section 4, we prove the interior Harnack inequality in Theorem \ref{thm:harnack}.
In Section 5,
we establish interior estimates for solutions to (\ref{eqn}).
Finally, in Section 6, 
we prove the H\"older
estimates in Corollary \ref{linftyhoelder} and Theorem \ref{mainthm}.

\section{Preliminaries}
In this section, we introduce some notations, definitions, 
and background results on the Monge-Amp\`ere equations and 
the linearized Monge-Amp\`ere equations
that will be used in this paper.

\subsection*{Notation}
We will use the following notations throughout the paper.
\begin{itemize}
\item $B_r(x):=\{y\in\mathbb{R}^n:|y-x|<r\}$,
\item $B_r:=B_r(0)$,
\item $u^{\pm}:=\max\{\pm u,0\}$,
\item $I_n:=n\times n$ identity matrix.
\item $\diam(E):=$ diameter of a set $E$.
\item $|\Omega|:=$ the Lebesgue measure of a Lebesgue measurable set $\Omega\subset\mathbb{R}^n$.
\end{itemize}

Unless otherwise stated,
our convex domains are assumed to have nonempty interior.

\begin{definition}[Sections]\label{def:sec} 
	Let $\varphi$ be a $C^1$ convex function in $\overline{\Omega}$.
	Then the \emph{section} of $\varphi$ centered at $x\in\overline{\Omega}$ with height $h>0$ is defined as
	\begin{equation*}
	\begin{aligned}
		S_\varphi(x,h)=\{y\in\overline{\Omega} :
		\varphi(y) < \varphi(x) + D\varphi(x)\cdot (y-x) + h \}
		\text{.}
	\end{aligned}
	\end{equation*}
\end{definition}

\begin{theorem}[John's lemma \cite{John}]\label{john}
	Let $\Omega\subset\mathbb{R}^n$ be a nonempty bounded convex domain.
	Then, there is an affine transformation $T:\mathbb{R}^n\rightarrow\mathbb{R}^n$ such that
	$B_1 \subset T^{-1}\Omega \subset B_n$.
\end{theorem}

\begin{definition}[Normalized convex sets]
An open convex set $K\subset\mathbb{R}^n$ is called \emph{normalized} if
$B_1\subset K\subset B_n$.
\end{definition}

We will use the following Monge-Amp\`ere Sobolev inequality.
It was proved by Tian-Wang \cite[Theroem 3.1]{masobolev} when $n\geq 3$,
and by Le \cite[Proposition 2.6]{SG_LMA} when $n=2$;
see also \cite[Theorem 14.15]{Le24}.
\begin{theorem}[Monge-Amp\`ere Sobolev inequality]\label{ma-sobolev}
    Let $\varphi$ be a $C^2$ convex function satisfying (\ref{maeqn}),
    and define $\Phi$ as in (\ref{cofdef}).
	Suppose $S_\varphi(x_0, 2h)\Subset\Omega$,
	and $S_\varphi(x_0,h)$ is a normalized section.
	Then for any $u\in C_c^\infty(S_\varphi(x_0,h))$,
	\begin{equation*}
	\begin{aligned}
		\norm{u}_{L^p(S_\varphi(x_0,h))}
		\leq C \left[ \int_{S_\varphi(x_0,h)} \Phi Du \cdot Du \,dx\right]^{1/2}
		\text{,}
	\end{aligned}
	\end{equation*}
	where
	\begin{enumerate}
		\item $p\in (2,\infty)$ and $C = C(p, \lambda, \Lambda)$ if $n = 2$, and
		\item $p = \frac{2n}{n-2}$ and $C = C(n,\lambda,\Lambda)$ if $n \geq 3$.
	\end{enumerate}
\end{theorem}

\begin{theorem}[Caffarelli's interior 
	$C^{1,\alpha}$ estimate \cite{CaffarelliC1a}]\label{intc1a}
	Let $\varphi$ be a strictly convex solution to the Monge-Amp\`ere equation
	$\det D^2\varphi = f$ in a convex domain $\Omega\subset\mathbb{R}^n$,
	where $\lambda\leq f\leq\Lambda$ for positive constants $\lambda$ and $\Lambda$.
	If $S_\varphi(x,h)\Subset\Omega$ is a normalized section, then
	for all $y,z\in S_\varphi(x,h/2)$, we have
	\begin{align*}
	|D\varphi(y)-D\varphi(z)| \leq C|y-z|^\alpha
	\text{,}
	\end{align*}
	where
	\begin{equation}\label{intc1aconst}
	\begin{aligned}
		C=C(n,\lambda,\Lambda)>0
		\quad\text{and }
		\alpha=\alpha(n,\lambda,\Lambda)>0
		\text{.}
	\end{aligned}
	\end{equation}
\end{theorem}

This $C^{1,\alpha}$ estimate implies that sections contain balls with the same center.

\begin{corollary}\label{intc1asec}
	With the same assumptions as in Theorem \ref{intc1a},
	if $t\leq h/2$ we have
	\begin{equation*}
	\begin{aligned}
		B_{ct^{1/(1+\alpha)}}(x) \subset S_\varphi(x, t)
		\text{,}
	\end{aligned}
	\end{equation*}
	where $\alpha$ is defined in (\ref{intc1aconst}) and $c=c(n,\lambda,\Lambda)>0$.
\end{corollary}

We will also use the interior $W^{2,1+\ee}$ estimate of 
De Philippis-Figalli-Savin \cite{DPFS} and Schmidt \cite{Schmidt}
for the Monge-Amp\`ere equation.
We will use the following formulation for compactly supported sections (see \cite[Corollary 6.26]{Le24}).
\begin{theorem}[Interior $W^{2,1+\ee}$ estimate]\label{intw2pest}
	Let $\Omega$ be a convex domain in $\mathbb{R}^n$.
	Let $\varphi$: $\Omega\rightarrow\mathbb{R}$
	be a continuous convex solution to the Monge-Amp\`ere equation 
	\begin{equation*}
	\begin{aligned}
		\det D^2\varphi = f \quad\text{in }\Omega\text{, }\quad
		0<\lambda\leq f\leq\Lambda
		\text{.}
	\end{aligned}
	\end{equation*}
	Suppose $S_\varphi(x_0,h)$ is a normalized section, and $S_\varphi (x_0, 2h)\Subset\Omega$.
	Then, for $\ee=\ee^*(n,\lambda,\Lambda)>0$ and $C=C(n,\lambda,\Lambda)>0$,
	we have
	\begin{equation*}
	\begin{aligned}
		\norm{D^2\varphi}_{L^{1+\ee}(S_\varphi(x_0,h))}\leq C
		\text{.}
	\end{aligned}
	\end{equation*}
\end{theorem}

We have the following volume estimates for sections (see \cite[Lemma 5.6(i)]{Le24}).
\begin{lemma}[Volume estimate for sections]\label{secvolest}
	Suppose $\varphi$ is a $C^1$ convex solution to $\lambda\leq \det D^2\varphi \leq \Lambda$
	for positive constants $\lambda$ and $\Lambda$ in $\Omega\subset\mathbb{R}^n$.
	If $S_\varphi (x,h)\Subset\Omega$, then
	\begin{equation*}
	\begin{aligned}
		c(\Lambda,n)h^{n/2}\leq|S_\varphi(x,h)|
		\leq C(\lambda,n)h^{n/2}
	\end{aligned}
	\end{equation*}
	for positive constants $c$ and $C$.
\end{lemma}

We will also use the following Harnack inequality for 
linearized Monge-Amp\`ere equations with drift from Le \cite[Theorem 1.1]{Le18}.
\begin{theorem}[Harnack inequality for linearized Monge-Amp\`ere equations]\label{harnackref}
Let $\Omega\subset\mathbb{R}^n$ be a bounded convex domain.
Assume that $\varphi$ satisfies (\ref{maeqn}),
and define $\Phi=(\Phi^{ij})_{1\leq i,j\leq n}$ as in (\ref{cofdef}).
Suppose that $v\geq 0$ is a $W_{\mathrm{loc}}^{2,n}(\Omega)$ solution of 
\begin{align}
\Phi^{ij} D_{ij} v + \bb \cdot Dv + cv = f
\end{align}
in a section $S:=S_\varphi(x_0, 2h)\Subset\Omega$, where 
$h\leq h_0$ for a positive, fixed $h_0$,
$f\in L_{\mathrm{loc}}^n(\Omega)$, $c\in L_{\mathrm{loc}}^n(\Omega)$, 
and $\bb\in L_{\mathrm{loc}}^\infty(\Omega;\mathbb{R}^n)$. 
Then
\begin{align}
\sup_{S_\varphi(x_0, h)} v
\leq C \left(\inf_{S_\varphi(x_0,h)} v + h^{1/2} \norm{f}_{L^n(S)} \right)
\text{,}
\end{align}
where $C$ is a positive constant depending on $n$, $\lambda$, $\Lambda$, $h_0$,
$\norm{\bb}_{L^\infty(S)}$, and $\norm{c}_{L^n(S)}$.
\end{theorem}

\begin{definition}[Subsolutions to equation (\ref{eqn}) in a domain $S$]
	Let $\Omega\subset\mathbb{R}^n$ be a bounded domain,
	and $S$ be a domain contained in $\Omega$.
	Suppose $\ff, \bbb, \bb\in L_{\mathrm{loc}}^\infty(\Omega;\mathbb{R}^n)
	\cap W_{\mathrm{loc}}^{1,n}(\Omega;\mathbb{R}^n)$ and
	$f\in L_{\mathrm{loc}}^n(\Omega)$. 
	We say that $u\in W^{1,2}(S)$ is a \emph{(weak) subsolution} to (\ref{eqn})
	if for all $v\in W_0^{1,2}(S)$ with $v\geq 0$ in $S$, we have
	\begin{equation}
	\begin{aligned}
		\int_S \Phi Du \cdot Dv \, dx + \int_S u\bbb\cdot Dv\, dx
		+\int_S (\bb\cdot Du) v \, dx \leq
		\int_S \ff \cdot Dv \, dx + \int_S fv \, dx
		\text{.}
	\end{aligned}
	\end{equation}
\end{definition}

\section{Global Estimates}

In this section, we prove global estimates for solutions
to equation (\ref{eqn}) with zero boundary data on sections in
Proposition \ref{lem:rescale}.
These estimates will be used to prove the Harnack inequality,
Theorem \ref{thm:harnack}, in Section 4.
The following is a brief outline of the steps leading to the
proof of Proposition \ref{lem:rescale}.

We begin with Lemma \ref{lemma1},
which provides an estimate for subsolutions $u$
that are nonpositive on the boundary of normalized sections.
By defining suitable test functions and using Moser iteration,
we derive an estimate for the $L^\infty$ norm of $u^+$ in terms of its $L^2$ norm.
In Lemma \ref{lemma2}, we obtain an $L^2$ bound for $w$ of the form $\log \frac{C}{C-u^+}$.
Next, in Lemma \ref{lemma3} we show that $w$ is a subsolution to a
linearized Monge-Amp\`ere equation of the form in (\ref{eqn}).
This gives global estimate for $u^+$ independent of the $L^2$ norm of $u$.
Applying Lemma \ref{lemma3} to $u$ and $-u$ gives Lemma \ref{lemma3-1},
which provides global estimates in normalized sections.
Finally, rescaling Lemma \ref{lemma3-1} gives us Proposition \ref{lem:rescale}.

We now proceed with the proof of the following lemma.         

\begin{lemma}\label{lemma1}
	Let $\varphi\in C^3(\Omega)$ be a convex function satisfying (\ref{maeqn}).
	Suppose $\ff, \bbb, \bb\in L_{\mathrm{loc}}^\infty(\Omega;\mathbb{R}^n)
	\cap W_{\mathrm{loc}}^{1,n}(\Omega;\mathbb{R}^n)$,
	$f\in L_{\mathrm{loc}}^n(\Omega)$, and $n/2<r\leq n$.
	Suppose $S = S_\varphi(x,t)$ is a normalized section, and $S_\varphi(x,2t)\Subset\Omega$.
	Suppose $u \in W^{1,2}(S)
	\cap C(\overline{S})$ 
	is a subsolution to (\ref{eqn}) in $S$ satisfying $u \leq 0$ on $\partial S$.
	Assume that
	\begin{enumerate}
		\item either $n= 2$, or 
		\item $n\geq 3$ and $\ee^*(n,\lambda,\Lambda) + 1 > \frac{n}{2}$,
		where $\ee^*$ is as in Theorem \ref{intw2pest}.
	\end{enumerate}
	Then,
	\begin{equation}\label{lem1est}
	\begin{aligned}
		\sup_S u^+\leq C\left(\norm{\ff}_{L^\infty(S)}+\norm{f}_{L^r(S)} + \norm{u^+}_{L^2(S)}\right)
		\text{,}
	\end{aligned}
	\end{equation}
	where 
	\begin{equation*}
	\begin{aligned}
		C &= C(n,\lambda,\Lambda, r, \ee^*, \norm{\bb}_{L^\infty(S)}, 
		\norm{\bbb}_{L^\infty(S)}, \norm{D^2\varphi}_{L^{1+\ee^*}(S)})
		\text{.}
	\end{aligned}
	\end{equation*}
\end{lemma}

\begin{proof} 
	We define the test function $v$ as in Gilbarg-Trudinger \cite[Section 8.5]{GT}.
	Set 
	\begin{align*}
		k = \norm{\ff}_{L^\infty(S)} + \norm{f}_{L^r(S)}
		\text{,}
	\end{align*} 
	and for $\beta\geq 1$ and $N\geq k$, define $H\in C^1([k,\infty))$ by
	\begin{equation*}
	\begin{aligned}
		H(z) = \begin{cases} 
			z^\beta - k^\beta & \text{ if } k\leq z\leq N\text{,} \\
			\beta N^{\beta -1}(z-N) + (N^\beta-k^\beta) & \text{ if } N < z \text{.}
		\end{cases}
	\end{aligned}
	\end{equation*}
	Let $w = u^+ + k \geq k$, and define
	\begin{equation*}
	\begin{aligned}
		v = G(w) := \int_k^w |H'(s)|^2 \, ds \geq 0
		\text{.}
	\end{aligned}
	\end{equation*}
	Then, using $v\in W_0^{1,2}(S)$ as a test function in (\ref{eqn}), we get
	\begin{equation}\label{1-ide1}
	\begin{aligned}
		\int_S \Phi Du \cdot Dv \, dx + \int_S u\bbb\cdot Dv\, dx
		+ \int_S (\bb\cdot Du) v \, dx \leq
		\int_S \ff \cdot Dv \, dx + \int_S fv \, dx
		\text{.}
	\end{aligned}
	\end{equation}
	
	Note that
	\begin{enumerate} 
		\item $Dv = G'(w) Dw = H'(w)^2 Dw$,
		\item $v$ and $Dv$ are supported on $\{u\geq 0\}$, 
			and on the set $\{u>0\}=\{v>0\}$, we have $Dw = Du = Du^+$, and
		\item $H'$ is increasing on $(k,\infty)$, 
			hence $G'$ is also increasing on $(k,\infty)$.
			Thus,
			\begin{align*}
			G(w)=\int_k^w G'(s)\, ds \leq wG'(w)
			\text{.}
			\end{align*}
	\end{enumerate}
	
	Now we estimate the terms in (\ref{1-ide1}) separately.
	Note that as $\varphi$ is convex and $\det D^2\varphi>0$ by (\ref{maeqn}), 
	$D^2\varphi$ is positive definite.
	Moreover, the largest eigenvalue of $D^2\varphi$ is bounded by $\Delta\varphi$.
	Therefore, we have, in the sense of symmetric matrices,
	\begin{equation*}
	\begin{aligned}
		\Phi = (\det D^2\varphi) (D^2\varphi)^{-1} \geq \frac{\det D^2\varphi}{\Delta\varphi}I_n
		\text{.}
	\end{aligned}
	\end{equation*}
	Hence for any $\eta\in\mathbb{R}^n$, we have, by (\ref{maeqn}),
	\begin{equation}\label{phiineq}
	\begin{aligned}
		\Phi \eta \cdot \eta 
		\geq \left( \frac{\det D^2\varphi}{\Delta\varphi}\right) |\eta|^2
		\geq \left( \frac{\lambda}{\Delta\varphi}\right) |\eta|^2
		\text{.}
	\end{aligned}
	\end{equation}
	
	Using the Cauchy-Schwarz inequality and (\ref{phiineq}), we get
	\begin{equation}\label{1-ide2-1}
	\begin{aligned}
		-\int_S (\bb \cdot Du) v \, dx
		&\leq \int_S G(w) |\bb \cdot Dw | \, dx 
		\leq \int_S wG'(w) |\bb \cdot Dw| \, dx \\
		&\leq \int_S \left(G'(w) \Phi Dw \cdot Dw \right)^{1/2} 
			\left(w^2 G'(w) \frac{\Delta\varphi}{\lambda} |\bb|^2 \right)^{1/2} \, dx \\
		&\leq\frac{1}{4}\int_S G'(w)\Phi Dw\cdot Dw \, dx 
		+ \int_S w^2 G'(w) \frac{\Delta\varphi}{\lambda} |\bb|^2 \, dx
		\text{.}
	\end{aligned}
	\end{equation}
	Similarly, recalling that $Dv$ is supported on $\{u\geq 0\}$,
	we have
	\begin{equation}
	\begin{aligned}
		-\int_S u\bbb \cdot Dv\, dx
		&= -\int_S G'(w) u \bbb\cdot Dw \, dx
		\leq \int_S G'(w) w |\bbb| |Dw| \, dx \\
		&\leq \int_S \left(G'(w) \Phi Dw \cdot Dw \right)^{1/2} 
			\left(w^2 G'(w) \frac{\Delta\varphi}{\lambda} |\bbb|^2 \right)^{1/2} \, dx \\
		&\leq\frac{1}{4}\int_S G'(w)\Phi Dw\cdot Dw \, dx 
		+ \int_S w^2 G'(w) \frac{\Delta\varphi}{\lambda} |\bbb|^2 \, dx
		\text{.}
	\end{aligned}
	\end{equation}
	By the same reason, we have
	\begin{equation}
	\begin{aligned}
		\int_S \ff\cdot Dv \, dx
		&= \int_S G'(w) \ff\cdot Dw \, dx\\
		&\leq \int_S \left(G'(w) \Phi Dw\cdot Dw \right)^{1/2}
		\left(G'(w) \frac{\Delta\varphi}{\lambda} |\ff|^2 \right)^{1/2} \, dx \\
		&\leq\frac{1}{4} \int_S G'(w) \Phi Dw\cdot Dw \, dx
		+ \int_S G'(w) \frac{\Delta\varphi}{\lambda} |\ff|^2 \, dx \\
		&\leq\frac{1}{4} \int_S G'(w) \Phi Dw\cdot Dw \, dx
		+ \int_S w^2 G'(w) \frac{\Delta\varphi}{\lambda}  \, dx
		\text{,}
	\end{aligned}
	\end{equation}
	where we used $w\geq \norm{\ff}_{L^\infty(S)}$,
	and because $w\geq k$,
	\begin{equation}\label{1-ide2-2}
	\begin{aligned}
		\int_S fv \, dx 
		&\leq \int_S |f|G(w) \, dx 
		\leq \int_S |f| wG'(w) \, dx \\
		&\leq \int_S \frac{|f|}{k} w^2 G'(w) \, dx
		\text{.}
	\end{aligned}
	\end{equation}
	
	Note that 
	\begin{equation*}
	\begin{aligned}
		\int_S \Phi Du \cdot Dv \, dx
		=\int_S G'(w) \Phi Dw\cdot Dw
		\text{.}
	\end{aligned}
	\end{equation*}
	Adding (\ref{1-ide2-1})--(\ref{1-ide2-2}) and invoking (\ref{1-ide1}), we obtain
	\begin{equation*}
	\begin{aligned}
		&\int_S G'(w) \Phi Dw\cdot Dw \, dx \\
		&\leq \frac{3}{4}\int_S G'(w) \Phi Dw\cdot Dw \, dx +  
		\int_S w^2 G'(w) \left[ \frac{\Delta\varphi}{\lambda}(1+|\bb|^2+|\bbb|^2) + \frac{|f|}{k} \right] \, dx
		\text{.}
	\end{aligned}
	\end{equation*}
	Hence
	\begin{equation}\label{1-ide3}
	\begin{aligned}
		\int_S G'(w) \Phi Dw\cdot Dw \, dx  
		\leq 4\int_S w^2G'(w) h \, dx
		\text{,}
	\end{aligned}
	\end{equation}
	where 
	\begin{equation*}
	\begin{aligned}
		h = \frac{\Delta\varphi}{\lambda}(1+|\bb|^2+|\bbb|^2) + \frac{|f|}{k}
		\text{.}
	\end{aligned}
	\end{equation*}
	Before moving to the next step, we estimate $h$.
	As $S$ is normalized, $|B_1|\leq |S|\leq |B_n|$.
	Therefore, for
	\begin{align*}
		q := \min \{ 1+\ee^*, r\} > \frac{n}{2}
	\end{align*}
	we have, by the H\"older inequality,
	\begin{equation}\label{hest}
	\begin{aligned}
		\norm{h}_{L^q(S)}
		&\leq \frac{1+\norm{\bb}_{L^\infty(S)}^2+\norm{\bbb}_{L^\infty(S)}^2}{\lambda}
		\norm{\Delta\varphi}_{L^q(S)} + \frac{\norm{f}_{L^q(S)}}{k} \\
		&\leq \frac{1+\norm{\bb}_{L^\infty(S)}^2+\norm{\bbb}_{L^\infty(S)}^2}{\lambda}
		\norm{\Delta\varphi}_{L^{1+\ee^*}(S)}|S|^{\frac{1+\ee^*-q}{q(1+\ee^*)}} 
		+ \frac{\norm{f}_{L^r(S)}}{k}|S|^{\frac{r-q}{qr}}\\
		&\leq \frac{1+\norm{\bb}_{L^\infty(S)}^2+\norm{\bbb}_{L^\infty(S)}^2}{\lambda}
		\norm{\Delta\varphi}_{L^{1+\ee^*}(S)}|B_n|^{\frac{1+\ee^*-q}{q(1+\ee^*)}} 
		+ |B_n|^{\frac{r-q}{qr}}
		\text{.}
	\end{aligned}
	\end{equation}
	Then, for 
	\begin{equation}\label{qdef}
	\begin{aligned}
		\widehat{q}:= \frac{2q}{q-1}
		\text{,}
	\end{aligned}
	\end{equation}
	we have, from the H\"older inequality
	\begin{equation}\label{1-ide3-1}
	\begin{aligned}
		\int_S w^2 G'(w) h \, dx 
		&=\int_S (wH'(w))^2 h \, dx \\
		&\leq \norm{h}_{L^q(S)} \norm{(wH'(w))^2}_{L^{\frac{q}{q-1}}(S)} \\
		&= \norm{h}_{L^q(S)} \norm{wH'(w)}_{L^{\widehat{q}}(S)}^2 
		\text{.}
	\end{aligned}
	\end{equation}
	As $u\leq 0$ on $\partial S$, $H(w)=0$ on $\partial S$ and
	the Monge-Amp\`ere Sobolev inequality, Theorem \ref{ma-sobolev}, implies
	\begin{equation}\label{1-ide3-2}
	\begin{aligned}
		\int_S G'(w) \Phi Dw \cdot Dw \, dx
		&= \int_S H'(w)^2 \Phi Dw \cdot Dw \, dx \\
		&= \int_S \Phi DH(w) \cdot DH(w) \, dx \\
		&\geq c_1(q, n,\lambda, \Lambda) \norm{H(w)}_{L^{\widehat{n}}(S)}^2
		\text{,}
	\end{aligned}
	\end{equation}
	where
	\begin{equation}\label{ndef}
	\begin{aligned}
		\widehat{n} = \begin{cases} 
			\frac{2n}{n-2} &\quad\text{if } n \geq 3 \text{,}\\
			2\widehat{q} &\quad\text{if } n = 2 \text{.}
		\end{cases}
	\end{aligned}
	\end{equation}
	Note that as $q > n/2$, we have $\widehat{n} > \widehat{q}$.
	From (\ref{1-ide3}), (\ref{1-ide3-1}), and (\ref{1-ide3-2}), we have
	\begin{equation}\label{1-ide4}
	\begin{aligned}
		\norm{H(w)}_{L^{\widehat{n}}(S)} \leq
		C_2(q, n,\lambda,\Lambda) \norm{h}_{L^q(S)}^{1/2} \norm{wH'(w)}_{L^{\widehat{q}}(S)}
		\text{.}
	\end{aligned}
	\end{equation}
	Letting $N\rightarrow\infty$, the terms in (\ref{1-ide4}) converge to 
	\begin{equation}\label{1-ide4-1}
	\begin{aligned}
		&\norm{H(w)}_{L^{\widehat{n}}(S)}\rightarrow \norm{w^\beta-k^\beta}_{L^{\widehat{n}}(S)}
		\text{,}\\
		&\norm{wH'(w)}_{L^{\widehat{q}}(S)}\rightarrow\norm{\beta w^\beta}_{L^{\widehat{q}}(S)}
		\text{.}
	\end{aligned}
	\end{equation}
	We also have
	\begin{equation}\label{1-ide4-2}
	\begin{aligned}
		\norm{k^\beta}_{L^{\widehat{n}}(S)}
		&= k^\beta |S|^{1/\widehat{n}}
		= |S|^{1/\widehat{n} - 1/\widehat{q}} \norm{k^\beta}_{L^{\widehat{q}}(S)}\\
		&\leq |B_1|^{1/\widehat{n}-1/\widehat{q}} \norm{k^\beta}_{L^{\widehat{q}}(S)} \\
		&\leq |B_1|^{1/\widehat{n}-1/\widehat{q}} \norm{w^\beta}_{L^{\widehat{q}}(S)}
		\text{.}
	\end{aligned}
	\end{equation}
	Because $\beta\geq 1$, from (\ref{1-ide4})--(\ref{1-ide4-2}) and (\ref{hest}),
	we conclude that
	\begin{equation}\label{1-ide5}
	\begin{aligned}
		&\norm{w^\beta}_{L^{\widehat{n}}(S)}
		\leq \widetilde{C} \beta \norm{w^\beta}_{L^{\widehat{q}}(S)}
		\text{,}
	\end{aligned}
	\end{equation}
	where
	\begin{equation*}
	\begin{aligned}
		\widetilde{C} &= C_2(n,\lambda,\Lambda, q)\norm{h}_{L^q(S)}^{1/2} + |B_1|^{1/\widehat{n}-1/\widehat{q}}\\
		&\leq\widetilde{C}(n,\lambda,\Lambda,r,\ee^*,
		\norm{\bb}_{L^\infty(S)}, \norm{\bbb}_{L^\infty(S)},
		\norm{D^2\varphi}_{L^{1+\ee^*}(S)})
		\text{.}
	\end{aligned}
	\end{equation*}
	Note that $\widetilde{C}$ is independent of $\beta$.
	
	We define
	\begin{equation*}
	\begin{aligned}
		\chi := \frac{\widehat{n}}{\widehat{q}} > 1 
		\text{,}
	\end{aligned}
	\end{equation*}
	and rewrite (\ref{1-ide5}) as
	\begin{equation}\label{1-ide6}
	\begin{aligned}
		\norm{w}_{L^{\beta\chi\widehat{q}}(S)}
		\leq (\widetilde{C}\beta)^{1/\beta} \norm{w}_{L^{\beta\widehat{q}}(S)}
		\text{.}
	\end{aligned}
	\end{equation}
	Setting $\beta=\chi^m \geq 1$ (for integer $m\geq 0$) in (\ref{1-ide6}), we get
	\begin{equation}\label{1-ide7}
	\begin{aligned}
		\norm{w}_{L^{\chi^{m+1}\widehat{q}}(S)}
		\leq \widetilde{C}^{\chi^{-m}} \chi^{m \chi^{-m}}
		\norm{w}_{L^{\chi^m\widehat{q}}(S)}
		\text{.}
	\end{aligned}
	\end{equation}
	Iterating (\ref{1-ide7}) yields
	\begin{equation}\label{1-ide8}
	\begin{aligned}
		\norm{w}_{L^\infty(S)}
		\leq \widetilde{C}^{\sum_{m\geq 0} \chi^{-m}} \chi^{\sum_{m\geq 0} m\chi^{-m}}
		\norm{w}_{L^{\widehat{q}}(S)}
		\text{.}
	\end{aligned}
	\end{equation}
	Because $w\geq u^+ \geq 0$ and
	\begin{equation*}
	\begin{aligned}
		\norm{w}_{L^{\widehat{q}}(S)}
		\leq \norm{w}_{L^\infty(S)}^{1-2/\widehat{q}}
		\norm{w}_{L^2(S)}^{2/\widehat{q}}
		\text{,}
	\end{aligned}
	\end{equation*}
	(\ref{1-ide8}) gives
	\begin{equation*}
	\begin{aligned}
		\sup_S u^+ \leq 
		\norm{w}_{L^\infty(S)} 
		&\leq C\norm{w}_{L^2(S)}
		\leq C(k + \norm{u^+}_{L^2(S)}) \\
		&= C\left(\norm{\ff}_{L^\infty(S)} + \norm{f}_{L^r(S)} + \norm{u^+}_{L^2(S)}\right)
		\text{,}
	\end{aligned}
	\end{equation*}
	where 
	\begin{equation*}
	\begin{aligned}
		C = C(n,\lambda,\Lambda, r, \ee^*, \norm{\bb}_{L^\infty(S)}, 
		\norm{\bbb}_{L^\infty(S)}, \norm{D^2\varphi}_{L^{1+\ee^*}(S)})
		\text{.}
	\end{aligned}
	\end{equation*}
	This completes the proof.
\end{proof}

Note that the $L^2$ norm of $u^+$ appears on the right-hand side of (\ref{lem1est}).
We will use a trick in Gilbarg-Trudinger \cite[Section 8.5]{GT} 
to eliminate this term.
We first prove the following lemma.

\begin{lemma}\label{lemma2}
	Let $\varphi\in C^3(\Omega)$ be a convex function satisfying (\ref{maeqn}).
	Suppose $\ff, \bbb, \bb\in L_{\mathrm{loc}}^\infty(\Omega;\mathbb{R}^n)
	\cap W_{\mathrm{loc}}^{1,n}(\Omega;\mathbb{R}^n)$, 
	$f\in L_{\mathrm{loc}}^n(\Omega)$, $n/2<r\leq n$, and $\dv \bbb \leq 0$. 
	Suppose $S = S_\varphi(x,t)$ is a normalized section and $S_\varphi(x,2t)\Subset\Omega$.
	Suppose $u \in W^{1,2}(S)
	\cap C(\overline{S})$ is a subsolution to (\ref{eqn}) in $S$ satisfying
	$u \leq 0$ on $\partial S$.
	Assume that
	\begin{enumerate}
		\item either $n= 2$, or 
		\item $n\geq 3$ and $\ee^*(n,\lambda,\Lambda) + 1 > \frac{n}{2}$,
		where $\ee^*$ is as in Theorem \ref{intw2pest}.
	\end{enumerate}
	Then, 
	\begin{equation}\label{wdef}
	\begin{aligned}
		w = \log \frac{M+k}{M+k-u^{+}}
		\quad\text{where }
		M = \sup_S u^+
		\quad\text{and }
		k = \norm{f}_{L^r(S)} + \norm{\ff}_{L^\infty(S)}
		\text{,}
	\end{aligned}
	\end{equation}
	satisfies
	\begin{equation}
	\begin{aligned}
		\norm{w}_{L^2(S)}
		\leq C(n, \lambda, \Lambda, r, \ee^*, \norm{\bb}_{L^\infty(S)}, 
		\norm{\bbb}_{L^\infty(S)},\norm{D^2\varphi}_{L^{1+\ee^*}(S)})
		\text{.}
	\end{aligned}
	\end{equation}
\end{lemma}

\begin{proof} 
	Set
	\begin{equation*}
	\begin{aligned}
		v:= \frac{u^+}{M+k-u^+}
		\text{.}
	\end{aligned}
	\end{equation*}
	Then $v\geq 0$, and $v\in W_0^{1,2}(S)\cap C(\overline{S})$.
	Because $u$ is a subsolution to (\ref{eqn}), we get
	\begin{equation*}
	\begin{aligned}
		\int_S \Phi Du \cdot Dv \, dx + \int_S u\bbb\cdot Dv\, dx
		+ \int_S (\bb\cdot Du) v \, dx \leq
		\int_S \ff \cdot Dv \, dx + \int_S fv \, dx
		\text{.}
	\end{aligned}
	\end{equation*}
	As $\dv\bbb\leq 0$ and $uv\geq 0$,
	\begin{equation*}
	\begin{aligned}
		\int_S u\bbb\cdot Dv\, dx 
		=\int_S \bbb\cdot D(uv)\, dx - \int_S v \bbb\cdot Du \, dx
		\geq -\int_S v\bbb\cdot Du \, dx
	\end{aligned}
	\end{equation*}
	and therefore, we have
	\begin{equation}\label{2-ide1}
	\begin{aligned}
		\int_S \Phi Du \cdot Dv \, dx 
		+ \int_S ((\bb-\bbb)\cdot Du) v \, dx \leq
		\int_S \ff \cdot Dv \, dx + \int_S fv \, dx
		\text{.}
	\end{aligned}
	\end{equation}
	Because
	\begin{equation}\label{2-ide2}
	\begin{aligned}
		Dv = \frac{M+k}{(M+k-u^+)^2} Du^+
		\text{,}
	\end{aligned}
	\end{equation}
	the left-hand side of (\ref{2-ide1}) becomes
	\begin{equation}\label{2-ide1-1}
	\begin{aligned}
		&\int_S \Phi Du \cdot Dv \, dx 
		+ \int_S ((\bb-\bbb)\cdot Du) v \, dx \\
		&=\int_S \frac{M+k}{(M+k-u^+)^2} \Phi Du^+\cdot Du^+ \, dx
		+ \int_S \frac{u^+ (\bb-\bbb)\cdot Du^+}{M+k-u^+} \, dx 
		\text{.}
	\end{aligned}
	\end{equation}
	We may also use (\ref{2-ide2}) to substitute $Dv$ in the right-hand side of (\ref{2-ide1}) to obtain
	\begin{equation}\label{2-ide1-2}
	\begin{aligned}
		\int_S \ff \cdot Dv \, dx + \int_S fv \, dx
		= 
		\int_S \frac{(M+k)\ff\cdot Du^+}{(M+k-u^+)^2} \, dx
		+ \int_S \frac{f u^+}{M+k-u^+} \, dx
		\text{.}
	\end{aligned}
	\end{equation}
	Putting (\ref{2-ide1}), (\ref{2-ide1-1}), and (\ref{2-ide1-2}) together, and dividing both sides by $M+k$, 
	we find
	\begin{equation}\label{2-ide3}
	\begin{aligned}
		\int_S \frac{\Phi Du^+\cdot Du^+}{(M+k-u^+)^2}\, dx
		\leq \int_S \frac{\ff\cdot Du^+}{(M+k-u^+)^2}\, dx +
		\int_S \frac{fu^+ +u^+(\bbb-\bb)\cdot Du^+}{(M+k)(M+k-u^+)} \, dx
		\text{.}
	\end{aligned}
	\end{equation}
	
	Now we estimate the terms in (\ref{2-ide3}) separately.
	First, from the Cauchy-Schwarz inequality and (\ref{phiineq}), we have
	\begin{equation}\label{2-ide3-1}
	\begin{aligned}
		\int_S \frac{\ff\cdot Du^+}{(M+k-u^+)^2}\, dx 
		&\leq
		\int_S \frac{(\Phi Du^+ \cdot Du^+)^{1/2} (\lambda^{-1} \Delta\varphi |\ff|^2)^{1/2}}{(M+k-u^+)^2}\, dx \\
		&\leq \frac{1}{4}\int_S \frac{\Phi Du^+ \cdot Du^+}{(M+k-u^+)^2}\, dx
		+ \int_S \frac{\lambda^{-1} \Delta\varphi |\ff|^2}{(M+k-u^+)^2}\, dx \\
		&\leq \frac{1}{4}\int_S \frac{\Phi Du^+ \cdot Du^+}{(M+k-u^+)^2}\, dx
		+ \int_S \lambda^{-1} \Delta\varphi \, dx
	\end{aligned}
	\end{equation}
	as $M+k-u^+ \geq k \geq \norm{\ff}_{L^\infty(S)}$.
	Next, using $k\geq \norm{f}_{L^r(S)}$ and the H\"older inequality, we estimate
	\begin{equation}\label{2-ide3-2}
	\begin{aligned}
		\int_S \frac{fu^+}{(M+k-u^+)(M+k)} \, dx
		&\leq \int_S \frac{|f|}{k} \times 1 \,dx
		\leq \norm{\frac{f}{k}}_{L^r(S)} \norm{1}_{L^{r/(r-1)}(S)} \\
		&\leq |S|^{\frac{r-1}{r}}
		\leq |B_n|^{\frac{r-1}{r}}
		\text{.}
	\end{aligned}
	\end{equation}
	Finally, from Cauchy-Schwarz inequality and (\ref{phiineq}), we estimate
	\begin{equation}\label{2-ide3-3}
	\begin{aligned}
		\int_S \frac{u^+ (\bbb-\bb) \cdot Du^+}{(M+k)(M+k-u^+)} \, dx
		&\leq \int_S |\bb-\bbb| \left| \frac{Du^+}{M+k-u^+} \right| \, dx \\
		&\leq \int_S \left\{ \frac{\Phi Du^+\cdot Du^+}{(M+k-u^+)^2} \right\}^{1/2} 
		\left\{ \frac{\Delta\varphi}{\lambda} |\bb-\bbb|^2 \right\}^{1/2} \, dx \\
		&\leq \frac{1}{4}\int_S \frac{\Phi Du^+ \cdot Du^+}{(M+k-u^+)^2}\, dx
		+ \int_S \frac{\Delta\varphi}{\lambda} |\bb-\bbb|^2 \, dx
		\text{.}
	\end{aligned}
	\end{equation}

	Combining (\ref{2-ide3}) with (\ref{2-ide3-1})--(\ref{2-ide3-3}) yields
	\begin{equation}\label{2-ide4}
	\begin{aligned}
		\frac{1}{2}\int_S \frac{\Phi Du^+\cdot Du^+}{(M+k-u^+)^2} \,dx
		&\leq |B_n|^{\frac{r-1}{r}} + 
		\frac{1+\left(\norm{\bb}_{L^\infty(S)}+\norm{\bbb}_{L^\infty(S)}\right)^2}{\lambda} \int_S \Delta\varphi\,dx \\
		&\leq C_0(n, r, \ee^*, 
		\lambda, \Lambda, \norm{\bb}_{L^\infty(S)}, \norm{\bbb}_{L^\infty(S)}, 
		\norm{D^2\varphi}_{L^{1+\ee^*}(S)})
		\text{.}
	\end{aligned}
	\end{equation}
	
	As $u\leq 0$ on $\partial S$, $w = 0$ on $\partial S$.
	Also, we have
	\begin{equation}\label{dw}
	\begin{aligned}
		Dw = \frac{Du^+}{M+k-u^+}
		\text{.}
	\end{aligned}
	\end{equation}
	Therefore, the left-hand side of (\ref{2-ide4}) can be estimated 
	using the H\"older inequality and 
	the Monge-Am\`pere Sobolev inequality in Theorem \ref{ma-sobolev}:
	\begin{equation}\label{2-ide4-1}
	\begin{aligned}
		&\frac{1}{2}\int_S \frac{\Phi Du^+\cdot Du^+}{(M+k-u^+)^2} \,dx
		=\frac{1}{2}\int_S \Phi Dw\cdot Dw \, dx \\
		&\geq
		\begin{cases}
		c_1 \norm{w}_{L^{\frac{2n}{n-2}}(S)}^2 
		\geq c_1 |S|^{-2/n} \norm{w}_{L^2(S)}^2
		\geq c_1 |B_n|^{-2/n} \norm{w}_{L^2(S)}^2
		&\text{ if } n\geq 3\text{,}\\
		c_1 \norm{w}_{L^4(S)}^2 
		\geq c_1 |B_n|^{-1/2} \norm{w}_{L^2(S)}^2
		&\text{ if } n=2\text{,}
		\end{cases}
	\end{aligned}
	\end{equation}
	where $c_1=c_1(n,\lambda,\Lambda)$.
	The conclusion of the lemma follows from (\ref{2-ide4}) and (\ref{2-ide4-1}).
\end{proof}

Now we obtain the following global $L^\infty$ estimate, 
independent of the $L^2$ norm of the solution $u$,
by showing that $w$ in (\ref{wdef}) is a subsolution to an equation of the same form as (\ref{eqn}).

\begin{lemma}\label{lemma3}
	Let $\varphi\in C^3(\Omega)$ be a convex function satisfying (\ref{maeqn}).
	Suppose $\ff, \bbb, \bb\in L_{\mathrm{loc}}^\infty(\Omega;\mathbb{R}^n)
	\cap W_{\mathrm{loc}}^{1,n}(\Omega;\mathbb{R}^n)$, 
	$f\in L_{\mathrm{loc}}^n(\Omega)$, $n/2<r\leq n$, and $\dv\bbb\leq 0$.
	Suppose $S = S_\varphi(x,t)$ is a normalized section and $S_\varphi(x,2t)\Subset\Omega$.
	Suppose $u \in W^{1,2}(S)
	\cap C(\overline{S})$ is a subsolution to (\ref{eqn}) in $S$ satisfying
	$u \leq 0$ on $\partial S$.
	Assume that 
	\begin{enumerate}
		\item either $n= 2$, or 
		\item $n\geq 3$ and $\ee^*(n,\lambda,\Lambda) + 1 > \frac{n}{2}$,
		where $\ee^*$ is from Theorem \ref{intw2pest}.
	\end{enumerate}
	Then,
	\begin{equation}
	\begin{aligned}
		\sup_S u^+ \leq C\left(\norm{\ff}_{L^\infty(S)} + \norm{f}_{L^r(S)} \right)
		\text{,}
	\end{aligned}
	\end{equation}
	where 
	\begin{equation*}
	\begin{aligned}
		C = C(n,\lambda,\Lambda, r, \ee^*, \norm{\bb}_{L^\infty(S)}, 
		\norm{\bbb}_{L^\infty(S)}, \norm{D^2\varphi}_{L^{1+\ee^*}(S)})
		\text{.}
	\end{aligned}
	\end{equation*}
\end{lemma}

\begin{proof}
	Let $w\in W_0^{1,2}(S)$ be as in (\ref{wdef}). 
	Then, using (\ref{dw}), we get
	\begin{equation}\label{3-ide1}
	\begin{aligned}
		-\bbb\cdot Dw
		&=-\frac{\bbb\cdot Du^+}{M+k-u^+} \\
		&= \frac{-\dv(\bbb u^+) + u^+\dv\bbb}{M+k-u^+}
		\leq \frac{-\dv(\bbb u^+)}{M+k-u^+}
	\end{aligned}
	\end{equation}
	as $\dv \bbb \leq 0$.
	We have in the weak sense,
	\begin{equation}\label{3-ide1-1}
	\begin{aligned}
		-\dv (\Phi Dw)=-\frac{\dv(\Phi Du^+)}{M+k-u^+}-\frac{\Phi Du^+\cdot Du^+}{(M+k-u^+)^2}
		\text{.}
	\end{aligned}
	\end{equation}

	From (\ref{dw}), (\ref{3-ide1}), and (\ref{3-ide1-1}), we get
	\begin{equation*}
	\begin{aligned}
		-\dv(\Phi Dw)+ (\bb-\bbb)\cdot Dw
		\leq \frac{-\dv(\Phi Du^+ +u^+\bbb)+\bb\cdot Du^+}{M+k-u^+}
		-\frac{\Phi Du^+\cdot Du^+}{(M+k-u^+)^2}
		\text{.}
	\end{aligned}
	\end{equation*}
	Combining this with (\ref{eqn}), we get
	\begin{equation*}
	\begin{aligned}
		&-\dv(\Phi Dw)+ (\bb-\bbb)\cdot Dw \\
		&\leq\frac{f-\dv\ff}{M+k-u^+}-\frac{\Phi Du^+\cdot Du^+}{(M+k-u^+)^2}\\
		&=-\dv\left(\frac{\ff}{M+k-u^+}\right)+\frac{f}{M+k-u^+}
		+\left(\frac{-\Phi Du^+\cdot Du^+ + \ff\cdot Du^+}{(M+k-u^+)^2}\right)
		\text{\quad in } \{u\geq 0\}\text{.}
	\end{aligned}
	\end{equation*}
	From (\ref{phiineq}) and the Cauchy-Schwarz inequality, we have in $S$
	\begin{equation*}
	\begin{aligned}
		-\Phi Du^+\cdot Du^+ +\ff\cdot Du^+
		&\leq -\frac{\lambda}{\Delta\varphi}|Du^+|^2 + \ff\cdot Du^+\\
		&\leq \frac{\Delta\varphi |\ff|^2}{4\lambda} \\
		&\leq \frac{\Delta\varphi (M+k-u^+)^2}{4\lambda} 
		\text{,}
	\end{aligned}
	\end{equation*}
	which implies
	\begin{equation*}
	\begin{aligned}
		-\dv(\Phi Dw)+ (\bb-\bbb)\cdot Dw 
		\leq -\dv\left(\frac{\ff}{M+k-u^+}\right)
		+\frac{f}{M+k-u^+}+\frac{\Delta\varphi}{4\lambda}
		\text{\quad in } \{u\geq 0\}\text{.}
	\end{aligned}
	\end{equation*}
	As $w=0$ outside $\{u\geq 0\}$, $w$ is a subsolution to
	\begin{equation}\label{3-ide2-1}
	\begin{aligned}
		-\dv(\Phi Dw)+\widetilde{\bb}\cdot Dw
		\leq -\dv\widetilde{\ff}+\widetilde{f}
		\quad\text{in }S
		\text{,}
	\end{aligned}
	\end{equation}
	where
	\begin{equation}
	\begin{aligned}
		\widetilde{\bb} &= \bb-\bbb \text{,}\\
		\widetilde{\ff} &= \frac{\ff}{M+k-u^+}\chi_{\{u\geq 0\}} \text{,}\quad\text{and}\\
		\widetilde{f} &= \left(\frac{f}{M+k-u^+} + \frac{\Delta\varphi}{4\lambda}\right)\chi_{\{u\geq 0\}}
		\text{.}
	\end{aligned}
	\end{equation}
	Recalling that $k=\norm{\ff}_{L^\infty(S)}+\norm{f}_{L^r(S)}$
	and $M=\sup_S u^+ \geq u^+$,
	we obtain
	\begin{equation}
	\begin{aligned}
		&\norm{\widetilde{\bb}}_{L^\infty(S)}
		\leq \norm{\bb}_{L^\infty(S)} +\norm{\bbb}_{L^\infty(S)}
		\text{,}\quad\text{and}\quad
		\norm{\widetilde{\ff}}_{L^\infty(S)} \leq 1
		\text{.}
	\end{aligned}
	\end{equation}
	For $\widetilde{r}:=\min\{r,1+\ee^*\}>n/2$, 
	using the H\"older inequality and the volume estimate 
	in Lemma \ref{secvolest}, we have	
	\begin{equation}\label{3-ide2-0}
	\begin{aligned}
		\norm{\widetilde{f}}_{L^{\widetilde{r}}(S)}
		&\leq \frac{\norm{f}_{L^{\widetilde{r}}(S)}}{k}
		+\frac{\norm{\Delta\varphi}_{L^{\widetilde{r}}(S)}}{4\lambda}\\
		&\leq C_1(n,r,\ee^*)\left(
			\frac{\norm{f}_{L^r(S)}}{k}+\frac{\norm{D^2\varphi}_{L^{1+\ee^*}(S)}}{4\lambda}
			\right) \\
		&\leq C_1(n,r,\ee^*)\left(
			1+\frac{\norm{D^2\varphi}_{L^{1+\ee^*}(S)}}{4\lambda}\right) \\
		&\leq C_2\left(n,r,\ee^*,\lambda,\norm{D^2\varphi}_{L^{1+\ee^*}(S)}\right)
		\text{.}
	\end{aligned}
	\end{equation}

	Combining (\ref{3-ide2-1})--(\ref{3-ide2-0}) and applying 
	Lemmas \ref{lemma1} and \ref{lemma2}, we get
	\begin{equation}\label{3-ide3-1}
	\begin{aligned}
		\sup_S w
		&\leq C_3\left(\norm{\widetilde{\ff}}_{L^\infty(S)}+
		\norm{f}_{L^{\widetilde{r}}(S)}+\norm{w}_{L^2(S)}\right)\\
		&\leq C_4(n,\lambda,\Lambda,r,\ee^*,\norm{\bb}_{L^\infty(S)},
		\norm{\bbb}_{L^\infty(S)},\norm{D^2\varphi}_{L^{1+\ee^*}(S)})
		\text{.}
	\end{aligned}
	\end{equation}
	Recalling that
	\begin{equation*}
	\begin{aligned}
		w = \log \frac{M+k}{M+k-u^+}
	\end{aligned}
	\end{equation*}
	and $M = \sup_S u^+$,
	we have
	\begin{equation}\label{3-ide3-2}
	\begin{aligned}
		\sup_S w = \log \frac{M+k}{k}
		\text{.}
	\end{aligned}
	\end{equation}
	Therefore, as $k=\norm{\ff}_{L^\infty(S)}+\norm{f}_{L^r(S)}$,
	the conclusion of the lemma follows from (\ref{3-ide3-1}) and (\ref{3-ide3-2}).
\end{proof}

By applying Lemma \ref{lemma3} to $u$ and $-u$, we obtain the following estimate.
\begin{lemma}\label{lemma3-1}
	Let $\varphi\in C^3(\Omega)$ be a convex function satisfying (\ref{maeqn}).
	Suppose $\ff, \bbb, \bb\in L_{\mathrm{loc}}^\infty(\Omega;\mathbb{R}^n)
	\cap W_{\mathrm{loc}}^{1,n}(\Omega;\mathbb{R}^n)$, 
	$f\in L_{\mathrm{loc}}^n(\Omega)$, $n/2<r\leq n$, and $\dv \bbb \leq 0$.
	Suppose $S = S_\varphi(x,t)$ is a normalized section and $S_\varphi(x,2t)\Subset\Omega$.
	Suppose $u \in W^{1,2}(S)
	\cap C(\overline{S})$ is a solution to (\ref{eqn}) in $S$ satisfying
	$u= 0$ on $\partial S$.
	Assume that
	\begin{enumerate}
		\item either $n= 2$, or 
		\item $n\geq 3$ and $\ee^*(n,\lambda,\Lambda) + 1 > \frac{n}{2}$,
		where $\ee^*$ is from Theorem \ref{intw2pest}.
	\end{enumerate}
	Then,
	\begin{equation}
	\begin{aligned}
		\norm{u}_{L^\infty(S)} \leq C\left(\norm{\ff}_{L^\infty(S)} + \norm{f}_{L^r(S)} \right)
	\end{aligned}
	\end{equation}
	where 
	\begin{equation*}
	\begin{aligned}
		C = C(n,\lambda,\Lambda, r, \ee^*, \norm{\bb}_{L^\infty(S)}, 
		\norm{\bbb}_{L^\infty(S)}, \norm{D^2\varphi}_{L^{1+\ee^*}(S)})
		\text{.}
	\end{aligned}
	\end{equation*}
\end{lemma}

Now, we rescale (\ref{eqn}) and apply Lemma \ref{lemma3-1} to obtain the following global estimate.

\begin{proposition}[Global $L^\infty$ estimate in normalized section]\label{lem:rescale}
	Let $\varphi\in C^3(\Omega)$ be a convex function satisfying (\ref{maeqn}).
	Suppose $\ff, \bbb, \bb\in L_{\mathrm{loc}}^\infty(\Omega;\mathbb{R}^n)
	\cap W_{\mathrm{loc}}^{1,n}(\Omega;\mathbb{R}^n)$, 
	$f\in L_{\mathrm{loc}}^n(\Omega)$, $n/2<r\leq n$, and $\dv \bbb \leq 0$.
	Suppose $S_\varphi (x_0, 2h_0)$ is a normalized section contained in $\Omega$, and $h\leq h_0$.
	Assume that $u \in W^{1,2}(S)
	\cap C(\overline{S})$ is a solution to (\ref{eqn})
	in $S=S_\varphi(x_0,h)$ satisfying $u=0$ on $\partial S$.
	Further assume that
	\begin{enumerate}
		\item either $n= 2$, or 
		\item $n\geq 3$ and $\ee^*(n,\lambda,\Lambda) + 1 > \frac{n}{2}$,
		where $\ee^*$ is from Theorem \ref{intw2pest}.
	\end{enumerate}
	Then,
	\begin{equation}
	\begin{aligned}
		\norm{u}_{L^\infty(S)} \leq C\left(\norm{\ff}_{L^\infty(S)} + \norm{f}_{L^r(S)} \right) h^\gamma
		\text{,}
	\end{aligned}
	\end{equation}
	where 
	\begin{equation*}
	\begin{aligned}
		&C=C(n,\lambda,\Lambda,r,\ee^*,\norm{\bb}_{L^\infty(S)}, 
		\norm{\bbb}_{L^\infty(S)})
		\text{,}\quad\text{and}\\
		&\gamma = \gamma(n, \lambda, \Lambda, r) > 0
		\text{.}
	\end{aligned}
	\end{equation*}
\end{proposition}

\begin{proof}
	We use the rescaling in Le (\cite[pp.20-22]{Twisted_Harnack}, \cite[Section 3.2]{SG_LMA}).
    By John's lemma, there is an affine transformation $Tx=A_h x + b_h$
	such that $B_1 \subset T^{-1}(S_\varphi(x_0,h)) \subset B_n$.
	We define the rescaled functions 
	\begin{equation}\label{rescalefcn}
	\begin{aligned}
		\widetilde{\varphi}(x) 
		&:= (\det A_h)^{-2/n} \varphi(Tx) \text{,}\\
		\widetilde{u}(x) 
		&:= u(Tx) \text{,}\\
		\widetilde{\ff}(x) 
		&:= (\det A_h)^{2/n} A_h^{-1} \ff (Tx)\text{,}\\
		\widetilde{\bb}(x)
		&:=(\det A_h)^{2/n} A_h^{-1} \bb (Tx)\text{,}\\
		\widetilde{\bbb}(x)
		&:=(\det A_h)^{2/n} A_h^{-1} \bbb (Tx)\text{,\quad and}\\
		\widetilde{f}(x)
		&:=(\det A_h)^{2/n} f(Tx)\\
	\end{aligned}
	\end{equation}
	on 
	\begin{equation}\label{rescaledom}
	\begin{aligned}
		\widetilde{S}:=T^{-1}(S_\varphi(x_0,h))=S_{\widetilde{\varphi}}(y_0,(\det A_h)^{-2/n}h)
		\text{,}
	\end{aligned}
	\end{equation}
	where $y_0=T^{-1}x_0$.
	Then, the rescaled functions satisfy the equation
	\begin{equation}\label{rescaledeq}
	\begin{aligned}
		-\dv(\widetilde{\Phi} D\widetilde{u} + \widetilde{u}\widetilde{\bbb})
		+ \widetilde{\bb}\cdot D\widetilde{u}
		= \widetilde{f}- \dv \widetilde{\ff}
		\quad\text{in }\widetilde{S}
		\text{.}
	\end{aligned}
	\end{equation}

	To apply Lemma \ref{lemma3-1} to $\widetilde{u}$,
	we estimate the rescaled functions.
	First, note that
	\begin{equation*}
	\begin{aligned}
		\det D^2\widetilde{\varphi}(x) = (\det D^2 \varphi)(Tx)
		\quad\text{in }\widetilde{S}\text{,}
	\end{aligned}
	\end{equation*}
	so that 	
	\begin{equation*}
	\begin{aligned}
		\lambda\leq\det D^2\widetilde{\varphi}\leq\Lambda
		\quad\text{in }\widetilde{S}\text{.}
	\end{aligned}
	\end{equation*}
	Furthermore, as $B_1\subset\widetilde{S}\subset B_n$, 
	we have from Lemma \ref{secvolest},
	\begin{equation}\label{4-ide1-1}
	\begin{aligned}
		c(n,\lambda,\Lambda) h^{n/2}\leq \det A_h \leq 
		C(n,\lambda,\Lambda) h^{n/2}
		\text{.}
	\end{aligned}
	\end{equation}
	From Corollary \ref{intc1asec}, we get
	\begin{equation}\label{4-ide1-2}
	\begin{aligned}
		\norm{A_h^{-1}}\leq \frac{n}{ch^{\frac{1}{1+\alpha}}}
		\leq C(n,\lambda,\Lambda)h^{-\frac{1}{1+\alpha}}
		\text{.}
	\end{aligned}
	\end{equation}
	
	Now, from (\ref{rescalefcn}), (\ref{4-ide1-1}), and (\ref{4-ide1-2}), we get
	\begin{equation}\label{4-ide3}
	\begin{aligned}
		\norm{\widetilde{\bb}}_{L^\infty(\widetilde{S})}
		&\leq (Ch^{n/2})^{2/n} Ch^{-\frac{1}{1+\alpha}} \norm{\bb}_{L^\infty(S)}
		\leq C(n,\lambda,\Lambda)
		h^{\frac{\alpha}{1+\alpha}} \norm{\bb}_{L^\infty(S)}
		\text{.}
	\end{aligned}
	\end{equation}
	Similarly, we also obtain
	\begin{equation}\label{4-ide2-1}
	\begin{aligned}
		\norm{\widetilde{\bbb}}_{L^\infty(\widetilde{S})}
		&\leq C(n,\lambda,\Lambda)
		h^{\frac{\alpha}{1+\alpha}} \norm{\bbb}_{L^\infty(S)}
		\text{,}\quad\text{and}\\
		\norm{\widetilde{\ff}}_{L^\infty(\widetilde{S})}
		&\leq C(n,\lambda,\Lambda)
		h^{\frac{\alpha}{1+\alpha}} \norm{\ff}_{L^\infty(S)}
		\text{.}
	\end{aligned}
	\end{equation}
	Finally, we get
	\begin{equation}\label{4-ide2-2}
	\begin{aligned}
		\norm{\widetilde{f}}_{L^r(\widetilde{S})}
		&=\left(\int_{\widetilde{S}} (\det A_h)^{2r/n} f^r(Tx) \,dx \right)^{1/r}\\
		&=\left(\int_S (\det A_h)^{(2r/n)-1} f^r(y) \,dy \right)^{1/r}\\
		&\leq\left(\int_S (Ch^{n/2})^{(2r/n)-1} f^r(y) \,dy \right)^{1/r}\\
		&=C(n,\lambda,\Lambda)^{2/n-1/r} h^{1-n/2r} \norm{f}_{L^r(S)}
		\text{.}
	\end{aligned}
	\end{equation}
	As $\partial\widetilde{S}=T^{-1}(\partial S)$, $\widetilde{u}=0$ on $\partial\widetilde{S}$.
	Therefore, we may apply Lemma \ref{lemma3-1} to $\widetilde{u}$ and combine it with 
	(\ref{4-ide2-1}) and (\ref{4-ide2-2}) to get 
	\begin{equation}\label{4-ide4-1}
	\begin{aligned}
		\norm{u}_{L^\infty(S)}
		&=\norm{\widetilde{u}}_{L^\infty(\widetilde{S})}\\
		&\leq \widetilde{C}\left(\norm{\widetilde{\ff}}_{L^\infty(\widetilde{S})}
			+\norm{\widetilde{f}}_{L^r(\widetilde{S})}\right) \\
		&\leq \widetilde{C}\left(
			C(n,\lambda,\Lambda)h^{\frac{\alpha}{1+\alpha}} \norm{\ff}_{L^\infty(S)}
			+C(n,\lambda,\Lambda,r)h^{1-n/2r}\norm{f}_{L^r(S)}\right)
			\text{,}
	\end{aligned}
	\end{equation}
	where
	\begin{equation}\label{4-ide4-2}
	\begin{aligned}
		\widetilde{C}
		&=C(n,\lambda,\Lambda,r,\ee^*,\norm{\widetilde{\bb}}_{L^\infty(\widetilde{S})},
		\norm{\widetilde{\bbb}}_{L^\infty(\widetilde{S})}, 
		\norm{D^2 \widetilde{\varphi}}_{L^{1+\ee^*}(\widetilde{S})})
				\text{.}
	\end{aligned}
	\end{equation}

	As $S_\varphi(x_0,h)$ is contained in a normalized section, we have
	\begin{equation}\label{4-ide4-22}
	\begin{aligned}
		h\leq C(n,\lambda,\Lambda)
		\text{.}
	\end{aligned}
	\end{equation}
	Therefore, we have
	\begin{equation}\label{4-ide4-3}
	\begin{aligned}
		\widetilde{C}
		\leq C(n,\lambda,\Lambda,r,\ee^*,\norm{\bb}_{L^\infty(S)},
		\norm{\bbb}_{L^\infty(S)})
		\text{.}
	\end{aligned}
	\end{equation}	
	Furthermore, the $L^\infty$ norms of $\widetilde{\bb}$, 
	$\widetilde{\bbb}$ are under control by 
	(\ref{4-ide3}) and (\ref{4-ide2-1}). 
	Finally, by the $W^{2,1+\ee}$ estimate in Theorem \ref{intw2pest}, we have
	\begin{equation*}
	\begin{aligned}
		\norm{D^2\widetilde{\varphi}}_{L^{1+\ee^*}(\widetilde{S})}
		\leq C(n,\lambda,\Lambda)
		\text{.}
	\end{aligned}
	\end{equation*}
	Combining (\ref{4-ide4-1}), (\ref{4-ide4-2}), and (\ref{4-ide4-22}),
	we have
	\begin{equation}\label{4-ide4-33}
	\begin{aligned}
		\norm{u}_{L^\infty(S)}
		\leq \widetilde{C}C(n,\lambda,\Lambda,r,\alpha)\left(
			\norm{\ff}_{L^\infty(S)}+\norm{f}_{L^r(S)}\right)
			h^{\gamma(n,r,\alpha)}
			\text{,}
	\end{aligned}
	\end{equation}
	where
	$$
	\gamma=\min\left\{1-\frac{n}{2r}, \frac{\alpha}{1+\alpha}\right\}
	\text{.}
	$$
	As $\alpha=\alpha(n,\lambda,\Lambda)$,
	the conclusion of the lemma follows from
	(\ref{4-ide4-3}) and (\ref{4-ide4-33}).
\end{proof}

\section{Harnack Inequality}

In this section,
we use the global estimate in Proposition \ref{lem:rescale} to prove
the Harnack inequality, Theorem \ref{thm:harnack}.
We begin by expressing an arbitrary solution of (\ref{eqn})
as the sum of solutions of a homogeneous equation
and an inhomogeneous equation with zero boundary data.
The inhomogeneous part can be bounded using Proposition \ref{lem:rescale},
while the homogeneous part can be bounded using the Harnack inequality
in Theorem \ref{harnackref}.
Combining these estimates yields the Harnack inequality in
normalized sections, Proposition \ref{homharnack}.
Rescaling Proposition \ref{homharnack} then gives the desired Harnack inequality
in Theorem \ref{thm:harnack}.

We will first prove the following proposition.

\begin{proposition}[Harnack inequality in normalized section]\label{homharnack} 
	Let $\varphi\in C^3(\Omega)$ be a convex function satisfying (\ref{maeqn}).
	Suppose that $\ff,\bbb,\bb \in L_{\mathrm{loc}}^\infty(\Omega;\mathbb{R}^n)
	\cap W_{\mathrm{loc}}^{1,n}(\Omega;\mathbb{R}^n)$,
	$f\in L_\mathrm{loc}^n(\Omega)$,
	$\dv\bbb\leq0$, and $n/2<r\leq n$.
	Suppose $S_\varphi(x, h_0)$ is a normalized section contained in $\Omega$, and $h\leq h_0/2$.
	Assume that $u \in W^{2,n}(S_\varphi(x,h))$ 
	is a nonnegative solution to (\ref{eqn}) in $S_\varphi(x,h)$.
	Further assume that
	\begin{enumerate}
		\item either $n= 2$, or 
		\item $n\geq 3$ and $\ee^*(n,\lambda,\Lambda) + 1 > \frac{n}{2}$,
		where $\ee^*$ is from Theorem \ref{intw2pest}.
	\end{enumerate}
	Then,
	\begin{equation*}
	\begin{aligned}
		\sup_{S_\varphi(x,h/2)} u
		\leq C\left(
			(\norm{\ff}_{L^\infty(S_\varphi(x,h))}+\norm{f}_{L^r(S_\varphi(x,h))})h^\gamma 
			+\inf_{S_\varphi(x,h/2)} u
		\right)
		\text{,}
	\end{aligned}
	\end{equation*}
	where 
	\begin{equation*}
	\begin{aligned}
		\gamma &=\gamma(n,\lambda,\Lambda,r)>0
		\text{,}\quad\text{and }\\
		C&=C(n,\lambda,\Lambda,r,\ee^*,\norm{\bb}_{L^\infty(S_\varphi(x,h))},
			\norm{\dv \bbb}_{L^n(S_\varphi(x,h))},
			\norm{\bbb}_{L^\infty(S_\varphi(x,h))})>0
		\text{.}
	\end{aligned}
	\end{equation*}
\end{proposition}

\begin{proof} 
	By \cite[Theorem 9.15]{GT}, we can find a solution 
	$u_0\in W^{2,n}(S)$ to 
	\begin{equation*}
	\begin{aligned}
		\begin{cases} 
			-\dv (\Phi Du_0 +u_0\bbb)+\bb\cdot Du_0
			=f-\dv\ff
			&\text{in }S:=S_\varphi(x,h)\text{,} \\
			u_0=0
			&\text{on }\partial S\text{.}
		\end{cases}
	\end{aligned}
	\end{equation*}
	Then $v=u-u_0$ satisfies $v\geq 0$ on $\partial S$, and is a solution to 
	\begin{equation*}
	\begin{aligned}
		-\dv(\Phi Dv + v\bbb) + \bb\cdot Dv = 0
		\quad\text{in }S\text{.}
	\end{aligned}
	\end{equation*}
	Observing that the equation above can be written as
	\begin{equation*}
	\begin{aligned}
		-\Phi^{ij}D_{ij} v + (\bb-\bbb)\cdot Dv - (\dv\bbb)v = 0
	\end{aligned}
	\end{equation*}
	and $\dv\bbb\leq 0$, we have $v\geq 0$ in $S$ by the maximum principle \cite[Theorem 9.1]{GT}.
	As $S_\varphi(x,2h)$ is contained in a normalized section,
	$S_\varphi(x,2h)\subset B_n$ and
	$h\leq C(n,\lambda,\Lambda)$ by Lemma \ref{secvolest}.
	Therefore, we can apply the Harnack inequality in Theorem \ref{harnackref} to get
	\begin{equation}\label{7-ide1-1}
	\begin{aligned}
		&\sup_{S_\varphi (x,h/2)}v \leq C_1\inf_{S_\varphi(x,h/2)} v
		\text{,}
	\end{aligned}
	\end{equation}
	where
	\begin{equation*}
	\begin{aligned}
		&C_1=C_1(n,\lambda,\Lambda,\norm{\bb}_{L^\infty(S)},
		\norm{\bbb}_{L^\infty(S)}, \norm{\dv\bbb}_{L^n(S)})
		\text{.}
	\end{aligned}
	\end{equation*}

	By applying the global estimate in Proposition \ref{lem:rescale} to $u_0$, we obtain
	\begin{equation}\label{7-ide1-2}
	\begin{aligned}
		&\sup_{S} |u_0|
		\leq C_2(\norm{\ff}_{L^\infty(S)}+\norm{f}_{L^r(S)})h^\gamma
		\text{,}
	\end{aligned}
	\end{equation}
	where
	\begin{equation*}
	\begin{aligned}
		C_2&=C_2(n,\lambda,\Lambda,r,\ee^*,\norm{\bb}_{L^\infty(S)}, 
		\norm{\bbb}_{L^\infty(S)})
		\text{,}\quad\text{and}\\
		\gamma &= \gamma(n,\lambda,\Lambda,r) > 0
		\text{.}
	\end{aligned}
	\end{equation*}
	As $v=u-u_0$, combining (\ref{7-ide1-1}) and (\ref{7-ide1-2}) completes the proof.
\end{proof}

We are now ready to prove Theorem \ref{thm:harnack}.

\begin{proof}[Proof of Theorem \ref{thm:harnack}] 
	We prove the theorem by using the rescaling scheme in the proof of Proposition \ref{lem:rescale}.
	Using John's lemma, we find an affine transformation 
	\begin{equation*}
	\begin{aligned}
		Tx=A_h x + b_h
	\end{aligned}
	\end{equation*}
	such that
	\begin{equation}\label{res-ide-A}
	\begin{aligned}
		TB_1\subset S:=S_\varphi(x_0,h)\subset TB_n
		\text{,}
	\end{aligned}
	\end{equation}
	Using the transformation $T$, we define the rescaled functions as in 
	(\ref{rescalefcn}), (\ref{rescaledom}).

	We start by estimating the matrix $A_h$.
	First, from (\ref{res-ide-A}) and Lemma \ref{secvolest}, we obtain the following bounds on $\det A_h$:
	\begin{equation}\label{res-ide-C}
	\begin{aligned}
		&|\det A_h|
		=\frac{|TB_1|}{|B_1|}\leq \frac{|S_\varphi(x_0,h)|}{|B_1|}
		\leq C_1(n,\lambda,\Lambda)h^{n/2}
		=:C_3(n,\lambda, \Lambda, h)
		\text{,\quad and}\\
		&|\det A_h|
		=\frac{|TB_n|}{|B_n|}
		\geq \frac{|S_\varphi(x_0,h)|}{|B_n|}
		\geq c_1(n,\lambda, \Lambda)h^{n/2}
		=:c_4(n,\lambda, \Lambda, h)
		\text{.}
	\end{aligned}
	\end{equation}
	We also have (see \cite[(5.6)]{Le24})
	\begin{equation}\label{res-ide-B}
	\begin{aligned}
		\norm{A_h^{-1}}
		\leq C_2 = 
		\widetilde{C}(n,\lambda,\Lambda, \diam(S_\varphi(x_0,2h)))h^{-n/2}
		\text{.}
	\end{aligned}
	\end{equation}
	
	Recall that from (\ref{rescaledom}),
	\begin{equation*}
	\begin{aligned}
		\widetilde{S}:=T^{-1}(S_\varphi(x_0,h))=S_{\widetilde{\varphi}}(y_0,(\det A_h)^{-2/n}h)
		\text{.}
	\end{aligned}
	\end{equation*}
	We now estimate the rescaled functions.
	From (\ref{rescalefcn}), (\ref{res-ide-B}) and (\ref{res-ide-C}), we have
	\begin{equation}\label{res-ide-D}
	\begin{aligned}
		\norm{\widetilde{\bb}}_{L^\infty(\widetilde{S})}
		&\leq C_3^{2/n}C_2 \norm{\bb}_{L^\infty(S)}
		\text{,}\\
		\norm{\widetilde{\bbb}}_{L^\infty(\widetilde{S})}
		&\leq C_3^{2/n}C_2 \norm{\bbb}_{L^\infty(S)}
		\text{,}\\
		\norm{\widetilde{\ff}}_{L^\infty(\widetilde{S})}
		&\leq C_3^{2/n}C_2 \norm{\ff}_{L^\infty(S)}
		\text{,\quad and}\\
		\norm{\widetilde{f}}_{L^r(\widetilde{S})}
		&=\left(\int_S (\det A_h)^{(2r/n)-1} f^r(y) \,dy \right)^{1/r}\\
		&\leq\left(\int_S C_3^{(2r/n)-1} f^r(y) \,dy \right)^{1/r}\\
		&=C_3^{2/n-1/r} \norm{f}_{L^r(S)}
		\text{.}
	\end{aligned}
	\end{equation}
	Also, as 
	\begin{equation}
	\begin{aligned}
		\dv \widetilde{\bbb}(x) = (\det A_h)^{2/n} \dv \bbb(Tx) \leq 0
		\text{,}
	\end{aligned}
	\end{equation}
	we have
	\begin{equation}\label{res-ide-E}
	\begin{aligned}
		\norm{\dv \widetilde{\bbb}}_{L^n(\widetilde{S})}
		&=\left(\int_{\widetilde{S}} (\det A_h)^{2} [(\dv\bbb)(Tx)]^n \,dx \right)^{1/n}\\
		&=\left(\int_S (\det A_h) [(\dv\bbb)(y)]^n \,dy \right)^{1/n}\\
		&\leq\left(\int_S C_3 [(\dv\bbb)(y)]^n \,dy \right)^{1/n}\\
		&=C_3^{1/n} \norm{\dv\bbb}_{L^n(S)}
		\text{.}
	\end{aligned}
	\end{equation}

	For $t\leq h/2$, setting
	\begin{equation*}
	\begin{aligned}
		\widetilde{t}=(\det A_h)^{-2/n}t \leq (\det A_h)^{-2/n} h/2
		\text{}
	\end{aligned}
	\end{equation*}
	gives
	\begin{equation}\label{res-ide-G}
	\begin{aligned}
		S_{\widetilde{\varphi}}(y_0,\widetilde{t})=T^{-1}S_\varphi(x_0,t)
		\text{.}
	\end{aligned}
	\end{equation}
	Then, $\widetilde{u}$ is a solution to the rescaled equation (\ref{rescaledeq}) in $S_{\widetilde{\varphi}}(y_0,2\widetilde{t})$.
	Applying Proposition \ref{homharnack} to $\widetilde{u}$, we get
	\begin{equation}\label{res-ide-F}
	\begin{aligned}
		\sup_{S_{\widetilde{\varphi}}(y_0, \widetilde{t})} \widetilde{u}
		\leq C_5
		\left\{
			\left(\norm{\widetilde{\ff}}_{L^\infty(\widetilde{S})} + 
			\norm{\widetilde{f}}_{L^r(\widetilde{S})}\right) 
			\widetilde{t}^\gamma
			+ \inf_{S_{\widetilde{\varphi}}(y_0, \widetilde{t})} \widetilde{u}
		\right\}
		\text{.}
	\end{aligned}
	\end{equation} 
	Here, the constants $C_5$ and $\gamma$ come from Proposition \ref{homharnack}:
	\begin{equation*}
	\begin{aligned}
		\gamma &=\gamma(n,\lambda,\Lambda,r)>0
		\text{,}\quad\text{and }\\
		C_5&=C_5(n,\lambda,\Lambda,r,\ee^*,\norm{\widetilde{\bb}}_{L^\infty(\widetilde{S})},
			\norm{\dv \widetilde{\bbb}}_{L^n(\widetilde{S})},
			\norm{\widetilde{\bbb}}_{L^\infty(\widetilde{S})})>0
		\text{.}
	\end{aligned}
	\end{equation*}
	Furthermore, the norms $\norm{\widetilde{\bb}}_{L^\infty(\widetilde{S})}$,
		$\norm{\dv \widetilde{\bbb}}_{L^n(\widetilde{S})}$,
		$\norm{\widetilde{\bbb}}_{L^\infty(\widetilde{S})}$,
		$\norm{\widetilde{\ff}}_{L^\infty(\widetilde{S})}$,
		and $\norm{\widetilde{f}}_{L^r(\widetilde{S})}$
	of the rescaled functions are under control by (\ref{res-ide-D}) and (\ref{res-ide-E}).
	Finally, $\widetilde{t}$ is controlled by $t$ through
	\begin{equation}\label{res-ide-H}
	\begin{aligned}
		\widetilde{t}\leq c_4^{-2/n} t
		\text{.}
	\end{aligned}
	\end{equation}
	Therefore, putting (\ref{res-ide-D}), (\ref{res-ide-E}), (\ref{res-ide-G}) and 
	(\ref{res-ide-H}) together, we obtain the conclusion of the theorem from (\ref{res-ide-F}).
\end{proof}

\section{Interior Estimates}
In this section,
we prove the interior estimate for solutions to (\ref{eqn})
in Lemma \ref{lemma5}.
This estimate will be used in the proofs of the H\"older estimates
in Corollary \ref{linftyhoelder} and Theorem \ref{mainthm} in Section 6.

We begin by defining suitable test functions and then applying Moser iteration.
This yields an estimate in Lemma \ref{lemma4} for the $L^\infty$ norm of
solutions $u$ to (\ref{eqn}),
involving its $L^{q^*}$ norm in a larger section,
where $q^*$ is a finite number.
Next, using a dilation argument from Le \cite[Theorem 15.4]{Le24} and rescaling,
we obtain the interior estimate in Lemma \ref{lemma5}.

We will first prove the following lemma.

\begin{lemma}[Interior estimate in normalized section]\label{lemma4}
	Let $\varphi\in C^3(\Omega)$ be a convex function satisfying (\ref{maeqn}).
	Suppose $\ff, \bb, \bbb\in W_{\mathrm{loc}}^{1,n}(\Omega;\mathbb{R}^n)
	\cap L_{\mathrm{loc}}^\infty(\Omega;\mathbb{R}^n)$, 
	$f\in L_{\mathrm{loc}}^n(\Omega)$, and $n/2<r\leq n$.
	Assume that $S_\varphi(x_0, 2t)\Subset\Omega$,
	and $S_\varphi(x_0,t)$ is a normalized section.
	Assume that $u \in W^{1,2}(S_\varphi(x_0,t))$ 
	is a nonnegative solution to (\ref{eqn}) in $S_\varphi(x_0,t)$.
	Further assume that
	\begin{enumerate}
		\item either $n= 2$, or 
		\item $n\geq 3$ and $\ee^*(n,\lambda,\Lambda) + 1 > \frac{n(n-1)}{2}$
		where $\ee^*$ is from Theorem \ref{intw2pest}.
	\end{enumerate}
	Then,
	\begin{equation*}
	\begin{aligned}
		\sup_{S_\varphi(x_0, t/2)}u
		\leq C(\norm{u}_{L^{q*}(S_\varphi(x_0,t))}+\norm{\ff}_{L^\infty(S_\varphi(x_0,t))}+\norm{f}_{L^r(S_\varphi(x_0,t))})
		\text{,}
	\end{aligned}
	\end{equation*}
	where 
	\begin{equation*}
	\begin{aligned}
		C = C(n,\lambda,\Lambda, r, \ee^*, \norm{\bb}_{L^\infty(S_\varphi(x_0,t))}, 
		\norm{\bbb}_{L^\infty(S_\varphi(x_0,t))})
		\text{,}\quad\text{and}\quad
		q^*=q^*(\ee^*,n,r)
		\text{.}
	\end{aligned}
	\end{equation*}
\end{lemma}

\begin{proof} 
	We argue as in Le \cite[pp.515--517]{Le24}.
	Let $\overline{u}= u+k$, 
	where 
	$$k = \norm{\ff}_{L^\infty(S_\varphi(x_0,t))} + \norm{f}_{L^r(S_\varphi(x_0,t))}\text{.}$$
	For $\eta\in C_c^1(S_\varphi(x_0,t))$ to be determined later 
	and $\beta\geq 0$, we use
	$v=\eta^2 \overline{u}^{\beta+1}\in C_c^1(S_\varphi(x_0,t))$ 
	as a test function in (\ref{eqn}) and extend it to be zero
	outside $S_\varphi(x_0,t)$ to obtain
	\begin{equation}\label{5-ide1}
	\begin{aligned}
		\int_\Omega \Phi Du\cdot Dv \, dx + \int_\Omega u\bbb\cdot Dv\, dx
		+\int_\Omega v\bb\cdot Du\, dx = \int_\Omega \ff\cdot Dv\, dx +\int_\Omega fv
		\text{.}
	\end{aligned}
	\end{equation}
	
	Because
	\begin{equation*}
	\begin{aligned}
		Dv=(\beta+1)\eta^2\overline{u}^\beta D\overline{u}+2\eta\overline{u}^{\beta+1}D\eta
		\quad\text{and }D\overline{u}=Du
		\text{,}
	\end{aligned}
	\end{equation*}
	the terms in (\ref{5-ide1}) become
	\begin{equation}\label{5-ide2-1}
	\begin{aligned}
		\int_\Omega \Phi Du\cdot Dv\, dx
		&=(\beta+1)\int_\Omega \eta^2 \overline{u}^\beta \Phi D\overline{u}\cdot D\overline{u}\, dx
		+2\int_\Omega \eta\overline{u}^{\beta+1} \Phi D\overline{u}\cdot D\eta\, dx \text{,}\\
		\int_\Omega u\bbb\cdot Dv\, dx
		&=(\beta+1)\int_\Omega \eta^2 u\overline{u}^\beta \bbb\cdot D\overline{u}\, dx
		+2\int_\Omega \eta u\overline{u}^{\beta+1} \bbb\cdot D\eta\, dx \text{,}\\
		\int_\Omega v\bb\cdot Du\, dx
		&=\int_\Omega \eta^2\overline{u}^{\beta+1} \bb\cdot D\overline{u}\, dx \text{,}\\
		\int_\Omega \ff\cdot Dv\, dx
		&=(\beta+1)\int_\Omega \eta^2\overline{u}^\beta \ff\cdot D\overline{u}\, dx
		+2\int_\Omega \eta\overline{u}^{\beta+1} \ff\cdot D\eta\, dx \text{,}\quad\text{and}\\
		\int_\Omega fv\, dx
		&=\int_\Omega \eta^2\overline{u}^{\beta+1} f\, dx
		\text{.}
	\end{aligned}
	\end{equation}
	We now estimate these terms.
	By the Cauchy-Schwarz inequality, we have
	\begin{equation}\label{5-ide10}
	\begin{aligned}
		-2\int_\Omega \eta\overline{u}^{\beta+1} \Phi D\overline{u}\cdot D\eta\, dx 
		\leq\frac{1}{8}\int_\Omega\eta^2\overline{u}^\beta\Phi D\overline{u}\cdot D\overline{u} \,dx
		+8\int_\Omega\overline{u}^{\beta+2}\Phi D\eta\cdot D\eta\, dx
		\text{.}
	\end{aligned}
	\end{equation}
	By the Cauchy-Schwarz inequality and (\ref{phiineq}), we get
	\begin{equation}
	\begin{aligned}
		&-(\beta+1)\int_\Omega \eta^2 u\overline{u}^\beta \bbb\cdot D\overline{u}\, dx \\
		&\leq(\beta+1)\int_\Omega \eta^2 \overline{u}^{\beta+1}|\bbb| |D\overline{u}|\, dx \\
		&\leq\frac{\beta+1}{8}\int_\Omega \eta^2 \overline{u}^{\beta}\Phi D\overline{u}\cdot D\overline{u}\, dx
		+2(\beta+1)\int_\Omega \eta^2\overline{u}^{\beta+2}\frac{\Delta\varphi}{\lambda} |\bbb|^2\, dx
		\text{.}
	\end{aligned}
	\end{equation}
	Similarly,
	\begin{equation}
	\begin{aligned}
		-2\int_\Omega \eta u\overline{u}^{\beta+1} \bbb\cdot D\eta\, dx
		&\leq 2\int_\Omega \eta\overline{u}^{\beta+2}|\bbb||D\eta|\, dx \\
		&\leq\int_\Omega \overline{u}^{\beta+2}\Phi D\eta\cdot D\eta\, dx
		+\int_\Omega \eta^2\overline{u}^{\beta+2} \frac{\Delta\varphi}{\lambda}|\bbb|^2 \, dx
		\text{,}
	\end{aligned}
	\end{equation}
	and
	\begin{equation}
	\begin{aligned}
		-\int_\Omega \eta^2\overline{u}^{\beta+1} \bb\cdot D\overline{u}\, dx
		\leq\frac{1}{8}\int_\Omega \eta^2\overline{u}^\beta\Phi D\overline{u}\cdot D\overline{u}\, dx
		+2\int_\Omega \eta^2\overline{u}^{\beta+2} \frac{\Delta\varphi}{\lambda}|\bb|^2\, dx
		\text{.}
	\end{aligned}
	\end{equation}
	As in (\ref{5-ide10}) and using 
	$\overline{u}\geq |\ff|$ in $S_\varphi(x_0,t)$, we have
	\begin{equation}
	\begin{aligned}
		&(\beta+1)\int_\Omega \eta^2\overline{u}^\beta \ff\cdot D\overline{u}\, dx \\
		&\leq (\beta+1)\int_\Omega 
		\left(\eta^2\overline{u}^\beta\Phi D\overline{u}\cdot D\overline{u}\right)^{1/2}
		\left(\eta^2\overline{u}^\beta\frac{\Delta\varphi}{\lambda}|\ff|^2\right)^{1/2}\, dx \\
		&\leq \frac{\beta+1}{8}\int_\Omega \eta^2\overline{u}^\beta\Phi D\overline{u}\cdot D\overline{u}\, dx
		+ 2(\beta+1)\int_\Omega \eta^2\overline{u}^\beta\frac{\Delta\varphi}{\lambda}|\ff|^2 \, dx \\
		&\leq \frac{\beta+1}{8}\int_\Omega \eta^2\overline{u}^\beta\Phi D\overline{u}\cdot D\overline{u}\, dx
		+ 2(\beta+1)\int_\Omega \eta^2\overline{u}^{\beta+2}\frac{\Delta\varphi}{\lambda} \, dx
		\text{,}
	\end{aligned}
	\end{equation}
	and
	\begin{equation}
	\begin{aligned}
		2\int_\Omega \eta\overline{u}^{\beta+1} \ff\cdot D\eta\, dx
		&\leq 2\int_\Omega \left(\overline{u}^{\beta+2}\Phi D\eta\cdot D\eta\right)^{1/2}
		\left(\eta^2\overline{u}^\beta \frac{\Delta\varphi}{\lambda}|\ff|^2\right)^{1/2}\,dx \\
		&\leq \int_\Omega \overline{u}^{\beta+2}\Phi D\eta\cdot D\eta\, dx
		+ \int_\Omega \eta^2\overline{u}^\beta \frac{\Delta\varphi}{\lambda}|\ff|^2\, dx \\
		&\leq \int_\Omega \overline{u}^{\beta+2}\Phi D\eta\cdot D\eta\, dx
		+ \int_\Omega \eta^2\overline{u}^{\beta+2} \frac{\Delta\varphi}{\lambda}\, dx
		\text{.}
	\end{aligned}
	\end{equation}
	Finally,
	\begin{equation}\label{5-ide2-0}
	\begin{aligned}
		\int_\Omega \eta^2\overline{u}^{\beta+1}f\,dx
		\leq \int_\Omega \eta^2\overline{u}^{\beta+2} \frac{|f|}{k} \, dx
		\text{.}
	\end{aligned}
	\end{equation}
	
	Now we put (\ref{5-ide1})--(\ref{5-ide2-0}) together.
	We use (\ref{5-ide2-1}) to substitute the integrals in (\ref{5-ide1});
	then, we apply the estimates in (\ref{5-ide10})--(\ref{5-ide2-0}).
	As $\beta\geq 0$ and each integral appearing on the right-hand sides
	of (\ref{5-ide10})--(\ref{5-ide2-0}) is nonnegative, we get
	\begin{equation}\label{5-ide3-0}
	\begin{aligned}
		&\frac{1}{2}(\frac{\beta}{2}+1)
		\int_\Omega \eta^2\overline{u}^\beta \Phi D\overline{u}\cdot D\overline{u}\, dx\\
		&\leq 10\left(
			\int_\Omega \overline{u}^{\beta+2}\Phi D\eta\cdot D\eta\, dx
			+\frac{\beta+2}{2}\int_\Omega \eta^2\overline{u}^{\beta+2}
			\left\{\frac{\Delta\varphi}{\lambda}(1+|\bb|^2+|\bbb|^2)+\frac{|f|}{k}\right\}\, dx
		\right)
		\text{.}
	\end{aligned}
	\end{equation}
	Because
	\begin{equation*}
	\begin{aligned}
		D(\overline{u}^{\beta/2+1}\eta)
		=\left(\frac{\beta}{2}+1\right)\overline{u}^{\beta/2}\eta D\overline{u}
		+\overline{u}^{\beta/2+1} D\eta
		\text{,}
	\end{aligned}
	\end{equation*}
	we have
	\begin{equation*}
	\begin{aligned}
		\Phi D(\overline{u}^{\beta/2+1}\eta)\cdot D(\overline{u}^{\beta/2+1}\eta)
		\leq 2\left[
			\left(\frac{\beta}{2}+1\right)^2\overline{u}^\beta \eta^2 \Phi D\overline{u} \cdot D\overline{u}
		+\overline{u}^{\beta+2} \Phi D\eta\cdot D\eta \right]
		\text{.}
	\end{aligned}
	\end{equation*}
	Therefore, (\ref{5-ide3-0}) implies that
	\begin{equation}\label{5-ide3}
	\begin{aligned}
		&\int_\Omega \Phi D(\overline{u}^{\beta/2+1}\eta)\cdot D(\overline{u}^{\beta/2+1}\eta) \, dx 
		\leq 128\left(\frac{\beta}{2}+1\right)^2
		\left[\int_\Omega \overline{u}^{\beta+2}\Phi D\eta\cdot D\eta\, dx\right. \\
			&\qquad\left.+\int_\Omega \eta^2\overline{u}^{\beta+2}
			\left\{\frac{\Delta\varphi}{\lambda}(1+|\bb|^2+|\bbb|^2)+\frac{|f|}{k}\right\}\, dx
		\right]
		\text{.}
	\end{aligned}
	\end{equation}
	
	Letting
	\begin{align*} 
		S_a:=S_\varphi(x_0, a)
		\text{,}
	\end{align*}
	we have, 
	from the Alexandrov Maximum Principle \cite[Theorem 3.12]{Le24} (also see \cite[(15.16)]{Le24}),
	\begin{align*}
	\dist (S_{\overline{r}}, \partial S_R)\geq c(n,\lambda,\Lambda)
	(R-{\overline{r}})^n 
	\text{\quad for } 0<{\overline{r}}<R\leq t
	\text{.}
	\end{align*}
	Hence, we may choose $\eta$ supported on $S_R$ so that $0\leq\eta\leq 1$,
	$\eta\equiv 1$ in $S_{\overline{r}}$, and
	\begin{equation}\label{gradest}
	\begin{aligned}
		|D\eta| \leq C_0(n,\lambda,\Lambda)(R-{\overline{r}})^{-n}
		\text{.}
	\end{aligned}
	\end{equation}
	We set 
	\begin{equation*}
	\begin{aligned}
		q=\min \left\{\frac{1+\ee^*}{n-1}, r\right\} > \frac{n}{2}
		\text{,}
	\end{aligned}
	\end{equation*}
	and define $\widehat{q}, \widehat{n}$ using (\ref{qdef}) and (\ref{ndef}). That is,
	\begin{equation*}
	\begin{aligned}
		\widehat{q}&:= \frac{2q}{q-1}
		\text{,\quad and }
		\widehat{n} := \begin{cases} 
			\frac{2n}{n-2} &\quad\text{if } n \geq 3 \text{,}\\
			2\widehat{q} &\quad\text{if } n = 2 \text{.}
		\end{cases}	\end{aligned}
	\end{equation*}
	Then, by the  Monge-Amp\`ere Sobolev inequality, Theorem \ref{ma-sobolev}, we have
	\begin{equation}\label{5-ide4-2}
	\begin{aligned}
		\int_\Omega \Phi D(\overline{u}^{\beta/2+1}\eta)\cdot D(\overline{u}^{\beta/2+1}\eta) \, dx 
		&=\int_{S_R} \Phi D(\overline{u}^{\beta/2+1}\eta)\cdot D(\overline{u}^{\beta/2+1}\eta) \, dx \\
		&\geq{c_1(n,\lambda,\Lambda, \widehat{q})}
		\norm{\overline{u}^{\beta/2+1}\eta}_{L^{\widehat{n}}(S_R)}^2 \\
		&\geq{c_1}\norm{\overline{u}^{\beta/2+1}}_{L^{\widehat{n}}(S_{\overline{r}})}^2 
		\text{.}
	\end{aligned}
	\end{equation}
	
	Because $D^2\varphi>0$, all of its eigenvalues are smaller than $\Delta\varphi$.
	Hence, 
	$$\Phi = (\det D^2\varphi)(D^2\varphi)^{-1}\leq (\Delta\varphi)^{n-1}I_n \text{.}$$ 
	Therefore, we have, from (\ref{gradest}),
	\begin{equation}\label{5-ide3-1}
	\begin{aligned}
		\int_\Omega \overline{u}^{\beta+2}\Phi D\eta\cdot D\eta\, dx
		&\leq \int_{S_R} \overline{u}^{\beta+2}(\Delta\varphi)^{n-1}|D\eta|^2 \, dx \\
		&\leq C_0^2(R-{\overline{r}})^{-2n}\int_{S_R}
		\overline{u}^{\beta+2}(\Delta\varphi)^{n-1}\, dx 
		\text{.} 
	\end{aligned}
	\end{equation}
	We also have
	\begin{equation}\label{5-ide3-2}
	\begin{aligned}
		&\int_\Omega \eta^2\overline{u}^{\beta+2}
			\left\{\frac{\Delta\varphi}{\lambda}(1+|\bb|^2+|\bbb|^2)+\frac{|f|}{k}\right\}\, dx \\
		&\leq \int_{S_R} \overline{u}^{\beta+2}
			\left\{\frac{\Delta\varphi}{\lambda}(1+|\bb|^2+|\bbb|^2)+\frac{|f|}{k}\right\}\, dx \\
		&\leq C_2(n,\lambda,\Lambda) (R-{\overline{r}})^{-2n} \int_{S_R} \overline{u}^{\beta+2}
			\left\{\frac{\Delta\varphi}{\lambda}(1+|\bb|^2+|\bbb|^2)+\frac{|f|}{k}\right\}\, dx
		\text{.}
	\end{aligned}
	\end{equation}
	If we define
	\begin{equation*}
	\begin{aligned}
		h := \frac{\Delta\varphi}{\lambda}(1+|\bb|^2+|\bbb|^2)+\frac{|f|}{k} + (\Delta\varphi)^{n-1}
		\text{,}
	\end{aligned}
	\end{equation*}
	then $h\in L^q(S_t)$.
	From (\ref{5-ide3-1}), (\ref{5-ide3-2}), and the H\"older inequality, 
	the right-hand side of (\ref{5-ide3}) is bounded by
	\begin{equation}\label{5-ide4-3}
	\begin{aligned}
		\text{RHS }(\ref{5-ide3})
		&\leq 128(C_0^2+C_2) 
		(\frac{\beta}{2}+1)^2 
		(R-{\overline{r}})^{-2n}\int_{S_R} \overline{u}^{\beta+2}h\, dx\\
		&\leq C_3(n,\lambda,\Lambda) 
		(\frac{\beta}{2}+1)^2 (R-{\overline{r}})^{-2n} 
		\norm{\overline{u}^{\beta/2+1}}_{L^{\widehat{q}}(S_R)}^2 \norm{h}_{L^q(S_t)}
		\text{.}
	\end{aligned}
	\end{equation}
	
	Combining (\ref{5-ide3}), (\ref{5-ide4-2}), and (\ref{5-ide4-3}) yields
	\begin{equation}\label{5-ide5}
	\begin{aligned}
		\norm{\overline{u}^{\beta/2+1}}_{L^{\widehat{n}}(S_{\overline{r}})}^2
		\leq C_4(n,\lambda,\Lambda, \widehat{q})\norm{h}_{L^q(S_t)}
		(R-{\overline{r}})^{-2n}(\frac{\beta}{2}+1)^2
		\norm{\overline{u}^{\beta/2+1}}_{L^{\widehat{q}}(S_R)}^2
		\text{.}
	\end{aligned}
	\end{equation}
	As $q>n/2$, $\widehat{n}>\widehat{q}$ and we may set
	\begin{equation*}
	\begin{aligned}
		\chi&:= \frac{\widehat{n}}{\widehat{q}} > 1
		\text{,}\quad\text{and }
		\gamma:=\widehat{q}(\frac{\beta}{2}+1)
		\text{.}
	\end{aligned}
	\end{equation*}
	Then, (\ref{5-ide5}) becomes 
	\begin{equation}\label{5-ide6}
	\begin{aligned}
		\norm{\overline{u}}_{L^{\gamma\chi}(S_{\overline{r}})}
		\leq \left(C_5(n,\lambda,\Lambda,\widehat{q})\norm{h}_{L^q(S_t)}
		(R-{\overline{r}})^{-2n}\gamma^2\right)^{\frac{q}{q-1}\frac{1}{\gamma}}
		\norm{\overline{u}}_{L^{\gamma}(S_R)}
		\text{.}
	\end{aligned}
	\end{equation}

	Define for each integer $j\geq 0$
	\begin{equation*}
	\begin{aligned}
		r_j = \frac{t}{2}+\frac{t}{2^{j+1}}
		\text{,}\quad\text{and }
		\gamma_j =\chi^j\widehat{q}
		\text{.}
	\end{aligned}
	\end{equation*}
	Setting $R=r_j$, ${\overline{r}}=r_{j+1}$, 
	and $\gamma=\gamma_j$ in (\ref{5-ide6}), we get
	\begin{equation}\label{5-ide7}
	\begin{aligned}
		\norm{\overline{u}}_{L^{\chi^{j+1}\widehat{q}}(S_{r_{j+1}})}
		\leq \left(2C_5 \widehat{q}^{2} \norm{h}_{L^q(S_t)} t^{-2n}
		2^{2n(j+2)}\chi^{2j}\right)^{\chi^{-j}/2}
		\norm{\overline{u}}_{L^{\chi^j}(S_{r_j})}
		\text{.}
	\end{aligned}
	\end{equation}
	Iterating (\ref{5-ide7}) yields
	\begin{equation}\label{5-ide8-1}
	\begin{aligned}
		\norm{\overline{u}}_{L^{\infty}(S_{t/2})}
		\leq \left( 2C_5\widehat{q}^{2} \norm{h}_{L^q(S_t)}t^{-2n}\right)^{\sum_{j\geq 0}\chi^{-j}/2}
		2^{\sum_{j\geq 0}n(j+2)\chi^{-j}} \chi^{\sum_{j\geq 0} j\chi^{-j}}
		\norm{\overline{u}}_{L^{\widehat{q}}(S_t)}
		\text{.}
	\end{aligned}
	\end{equation}
	As $S_\varphi(x_0, t)$ is normalized, we have from Lemma \ref{secvolest},
	\begin{equation}\label{5-ide8-2}
	\begin{aligned}
		t^{-1} \leq C_6(n,\lambda,\Lambda)
		\text{.}
	\end{aligned}
	\end{equation}
	Finally, the $W^{2,1+\ee}$ estimate in Theorem \ref{intw2pest} implies
	\begin{equation}\label{5-ide8-3}
	\begin{aligned}
		\norm{h}_{L^q(S_t)}
		&\leq
		C_7(n)\frac{1+\norm{\bb}_{L^\infty(S_t)}^2
		+\norm{\bbb}_{L^\infty(S_t)}^2}{\lambda}
		\norm{D^2\varphi}_{L^q(S_t)}\\
		&\quad+ \frac{\norm{f}_{L^q(S_t)}}{k} 
		+C_7(n)\norm{D^2\varphi}_{L^{{q}{(n-1)}}(S_t)}^{n-1}\\
		&\leq C_8(n,\ee^*,r,\lambda,\Lambda,\norm{\bb}_{L^\infty(S_t)},\norm{\bbb}_{L^\infty(S_t)}, 
		\norm{D^2\varphi}_{L^{1+\ee^*}(S_t)}) \\
		&\leq C_9(n,\ee^*,r,\lambda,\Lambda,\norm{\bb}_{L^\infty(S_t)},\norm{\bbb}_{L^\infty(S_t)}) 
		\text{.}
	\end{aligned}
	\end{equation}
	The conclusion of the lemma follows from (\ref{5-ide8-1})--(\ref{5-ide8-3}).
\end{proof}

Now, we rescale (\ref{eqn}) as we did in the proof of Proposition \ref{lem:rescale},
and apply the result in Lemma \ref{lemma4}.
Using the estimates from the proof of Theorem \ref{thm:harnack}, 
we then argue as in Le \cite[Theorem 15.4]{Le24} to obtain the following interior estimates in general sections.

\begin{lemma}[Interior estimate in general section]\label{lemma5}
	Let $\varphi\in C^3(\Omega)$ be a convex function satisfying (\ref{maeqn}).
	Suppose $\ff, \bb, \bbb\in W_{\mathrm{loc}}^{1,n}(\Omega;\mathbb{R}^n)
	\cap L_{\mathrm{loc}}^\infty(\Omega;\mathbb{R}^n)$, 
	$f\in L_{\mathrm{loc}}^n(\Omega)$, and $n/2<r\leq n$.
	Assume that $S_\varphi(x_0, 2h)\Subset\Omega$ 
	and $u \in W^{1,2}(S_\varphi(x_0,h))$ 
	is a nonnegative solution to (\ref{eqn}) in $S_\varphi(x_0,2h)$.
	Further assume that
	\begin{enumerate}
		\item either $n= 2$, or 
		\item $n\geq 3$ and $\ee^*(n,\lambda,\Lambda) + 1 > \frac{n(n-1)}{2}$
		where $\ee^*$ is from Theorem \ref{intw2pest}.
	\end{enumerate}
	Then,
	\begin{equation}\label{intl2est}
	\begin{aligned}
		\sup_{S_\varphi(x_0, h/2)}u
		\leq C(h^{-\frac{n}{4}}\norm{u}_{L^{2}(S_\varphi(x_0,h))}+
		h^{1-\frac{n}{2}}\norm{\ff}_{L^\infty(S_\varphi(x_0,h))}+
		h^{1-\frac{n}{2r}}\norm{f}_{L^r(S_\varphi(x_0,h))})
	\end{aligned}
	\end{equation}
	where 
	\begin{equation*}
	\begin{aligned}
		C = C(n,\lambda,\Lambda, r, \ee^*, 
		\norm{\bb}_{L^\infty(S_\varphi(x_0,h))}, 
		\norm{\bbb}_{L^\infty(S_\varphi(x_0,h))}, h
		,\diam (S_\varphi (x,2h)))>0
		\text{.}
	\end{aligned}
	\end{equation*}
\end{lemma}

\begin{proof}
We rescale $S=S_\varphi (x_0,h)$ as in the proof of Theorem \ref{thm:harnack}, 
so that $B_1\subset T^{-1}S_\varphi(x_0,h)\subset B_n$.
We will use $\widetilde{C}$ and the numbered constants $C_n$
to denote the same constants from the proof of Theorem \ref{thm:harnack}
throughout the proof of this lemma.

For $\widetilde{h}:=(\det A_h)^{-2/n}h$, we have the rescaled equation (\ref{rescaledeq})
in $\widetilde{S}=S_{\widetilde{\varphi}}(y_0, \widetilde{h})$.
Applying Lemma \ref{lemma4} to $\widetilde{u}$, we get
\begin{equation}\label{res-ide2-11}
\begin{aligned}
	\sup_{S_{\widetilde{\varphi}}(y_0, \widetilde{h}/2)}\widetilde{u}
	\leq D_1(\norm{\widetilde{u}}_{L^{q*}(S_{\widetilde{\varphi}}(y_0, \widetilde{h}))}+
	\norm{\widetilde{\ff}}_{L^\infty(S_{\widetilde{\varphi}}(y_0,\widetilde{h}))}
	+\norm{\widetilde{f}}_{L^r(S_{\widetilde{\varphi}}(y_0,\widetilde{h}))})
	\text{,}
\end{aligned}
\end{equation}
where $D_1>0$ depends on 
$n$, $\lambda$, $\Lambda$, $r$, $\ee^*$, 
$\norm{\widetilde{\bb}}_{L^\infty(\widetilde{S})}$, 
and $\norm{\widetilde{\bbb}}_{L^\infty(\widetilde{S})}$.

Using the expression for $C_2$ in (\ref{res-ide-B}) and $C_3$ in (\ref{res-ide-C}), 
we use the estimates (\ref{res-ide-D}) to estimate the norms of the rescaled functions:
\begin{equation}\label{res-ide2-1}
\begin{aligned}
	\norm{\widetilde{\bb}}_{L^\infty(S_{\widetilde{\varphi}}(y_0,\widetilde{h}))}
	&\leq \left({C_1}h^{n/2}\right)^{2/n}(\widetilde{C}h^{-n/2}) \norm{\bb}_{L^\infty(S_\varphi(x_0,h))}\\
	&= C_1^{2/n}\widetilde{C}h^{1-n/2} \norm{\bb}_{L^\infty(S_\varphi(x_0,h))}
	\text{,}\\
	\norm{\widetilde{\bbb}}_{L^\infty(S_{\widetilde{\varphi}}(y_0,\widetilde{h}))}
	&\leq C_1^{2/n}\widetilde{C}h^{1-n/2} \norm{\bbb}_{L^\infty(S_\varphi(x_0,h))}
	\text{,}\\
	\norm{\widetilde{\ff}}_{L^\infty(S_{\widetilde{\varphi}}(y_0,\widetilde{h}))}
	&\leq C_1^{2/n}\widetilde{C}h^{1-n/2} \norm{\ff}_{L^\infty(S_\varphi(x_0,h))}
	\text{,\quad and}\\
	\norm{\widetilde{f}}_{L^r(S_{\widetilde{\varphi}}(y_0,\widetilde{h}))}
	&= (C_1 h^{n/2})^{2/n-1/r} \norm{f}_{L^r(S_\varphi(x_0,h))} \\
	&= C_1^{2/n-1/r}h^{1-n/2r} \norm{f}_{L^r(S_\varphi(x_0,h))} 
	\text{.}
\end{aligned}
\end{equation}
We also have (see \cite[Lemma 15.2(iii)]{Le24})
\begin{equation}
\begin{aligned}
	\norm{\widetilde{u}}_{L^{q^*}(S_{\widetilde{\varphi}}(y_0, \widetilde{h}))}
	\leq D_2(n,\lambda,\Lambda, q^*)h^{-n/2q^*}\norm{u}_{L^{q^*}(S_\varphi(x_0,h))}
\end{aligned}
\end{equation}
for $q^* = q^*(\ee^*,n,r)$, and
\begin{equation}\label{res-ide2-0}
\begin{aligned}
	\sup_{S_{\widetilde{\varphi}}(y_0, \widetilde{h}/2)}\widetilde{u}
	=\sup_{S_\varphi(x_0, h/2)} u
	\text{.}
\end{aligned}
\end{equation}

The $L^\infty$ norms of $\widetilde{\bb}$ and $\widetilde{\bbb}$ are under control from (\ref{res-ide2-1}).
Hence, from (\ref{res-ide2-11})--(\ref{res-ide2-0}) we have,
\begin{equation}\label{qineq}
\begin{aligned}
	\sup_{S_\varphi(x_0, h/2)}u
	\leq D_3(h^{-\frac{n}{2q^*}}\norm{u}_{L^{q^*}(S_\varphi(x_0,h))}+
	h^{1-\frac{n}{2}}\norm{\ff}_{L^\infty(S_\varphi(x_0,h))}+
	h^{1-\frac{n}{2r}}\norm{f}_{L^r(S_\varphi(x_0,h))})
	\text{,}
\end{aligned}
\end{equation}
where $D_3$ depends on 
	$n$, $\lambda$, $\Lambda$, $r$, $\ee^*$, 
	$\norm{(\bb,\bbb)}_{L^\infty(S_\varphi(x_0,h))}$, 
	$h$, and $\diam(S_\varphi(x_0,2h))$.
We can now use (\ref{qineq}) to argue as in Le \cite[pp.519--521]{Le24}
(see also Han-Lin \cite[pp.75--76]{Han-Lin})
to obtain (\ref{intl2est}).
This gives the conclusion of the Lemma.
\end{proof}

\begin{remark}
In fact, following the arguments cited above,
we can obtain (\ref{intl2est}) with the $L^2$ norm of $u$ replaced by
the $L^p$ norm of $u$, for any $p>0$.
\end{remark}

\section{Interior H\"older Estimates}
In this section, we prove the interior H\"older estimates in 
Corollary \ref{linftyhoelder}
and Theorem \ref{mainthm}.
We start by combining the Harnack inequality in Theorem \ref{thm:harnack}
and the global estimate in Proposition \ref{lem:rescale}
to prove Corollary \ref{linftyhoelder}.

\begin{proof}[Proof of Corollary \ref{linftyhoelder}]
	Let $\osc (g,E):=\sup_E g-\inf_E g$.
	It is sufficient (see \cite[pp.523--524]{Le24}) to prove the oscillation estimate
	\begin{equation}\label{oscest}
	\begin{aligned}
		\osc(u,S_\varphi(x_0, h)) \leq C_0 \left(
			\norm{u}_{L^\infty(S_\varphi(x_0, h_0))}+
			\norm{\ff}_{L^\infty(S_\varphi(x_0, 2h_0))}+\norm{f}_{L^r(S_\varphi(x_0, 2h_0))}
		\right)h^{\gamma_0}
	\end{aligned}
	\end{equation}
	for all $h\leq h_0$, 
	where the positive constants $C_0$ and $\gamma_0$ 
	have the same dependency as $C$ and $\gamma$ stated in the Corollary.

	As in Le \cite[pp.284--285]{SG_LMA}, we break up the solution $u=v+w$ in $S_\varphi(x_0,h)$,
	$h\leq h_0$, where $v,w\in W^{2,n}(S_\varphi(x_0, h))$
	are solutions to 
	\begin{equation*}
	\begin{aligned}
		\begin{cases}
		    -\dv(\Phi Dv + v\bbb)+\bb\cdot Dv = f-\dv\ff
			&\text{in }S_\varphi(x_0,h) \text{,} \\
			v=0 &\text{on }\partial S_\varphi(x_0,h) \text{,}
		\end{cases}
	\end{aligned}
	\end{equation*}
	and
	\begin{equation*}
	\begin{aligned}
		\begin{cases}
		    -\dv(\Phi Dw + w\bbb)+\bb\cdot Dw = 0
			&\text{in }S_\varphi(x_0,h) \text{,}\\
			w=u &\text{on }\partial S_\varphi(x_0,h)
			\text{.}
		\end{cases}
	\end{aligned}
	\end{equation*}
	Such $u$ and $v$ exist as a consequence of \cite[Theorem 9.15]{GT}.
	
	We now rescale $S_\varphi(x_0, 2h_0)$ as in the proof of Theorem \ref{thm:harnack},
	so that $B_1\subset T^{-1}S_\varphi(x_0,2h_0)\subset B_n$
	for $Tx = A_{2h_0}x + b_{2h_0}$.
	We define the rescaled functions using (\ref{rescalefcn}),
	and set $\widetilde{v}(x):=v(Tx)$.
	Applying the global estimate in Proposition \ref{lem:rescale}
	to $\widetilde{v}$, we get 
	\begin{equation}\label{6-ide1-11}
	\begin{aligned}
		\norm{\widetilde{v}}_{L^\infty 
		(S_{\widetilde{\varphi}}(y_0, \widetilde{h}))}\leq
		\widetilde{C}_1
		\left(\norm{\widetilde{\ff}}_{L^\infty(S_{\widetilde{\varphi}}(y_0, \widetilde{h}))}+
		\norm{\widetilde{f}}_{L^r(S_{\widetilde{\varphi}}(y_0, \widetilde{h}))}\right)
		\widetilde{h}^{\widetilde{\gamma}_1}
		\text{,}
	\end{aligned}
	\end{equation}
	where 
	\begin{equation*}
	\begin{aligned}
		\widetilde{h}&:=(\det A_{2h_0})^{-2/n}h\text{,}\\
		\widetilde{C}_1 
		&= \widetilde{C}_1\left(n,\lambda,\Lambda,r,\ee^*,
		\norm{\widetilde{\bbb}}_{L^\infty(S_{\widetilde{\varphi}}(y_0, \widetilde{h}))},
		\norm{\widetilde{\bb}}_{L^\infty(S_{\widetilde{\varphi}}(y_0, \widetilde{h}))}\right)
		\text{,\quad and}\\
		{\gamma}_1 &= {\gamma}_1(n, \lambda, \Lambda, r) > 0
		\text{.}
	\end{aligned}
	\end{equation*}
	The $L^\infty$ norms of $\widetilde{\ff}$, $\widetilde{\bb}$,
	$\widetilde{\bbb}$,
	and the $L^r$ norm of $\widetilde{f}$
	are under control as in (\ref{res-ide2-1}).
	Also, by (\ref{res-ide-C}), 
	$$\widetilde{h}\leq C(n,\lambda,\Lambda, h_0)h\text{.}$$
	Hence, from (\ref{6-ide1-11}), we get
	\begin{equation}\label{6-ide1-1}
	\begin{aligned}
		\osc(v, S_\varphi(x_0, h/2))\leq
		2\norm{v}_{L^\infty(S_\varphi(x_0,h))}\leq
		C_1(\norm{\ff}_{L^\infty(S_\varphi(x_0, 2h_0))}+\norm{f}_{L^r(S_\varphi(x_0, 2h_0))})h^{\gamma_1}
		\text{,}
	\end{aligned}
	\end{equation}
	where 
	\begin{equation*}
	\begin{aligned}
		C_1 &= C_1(n,\lambda,\Lambda,r,\ee^*,\norm{\bb}_{L^\infty(S_\varphi(x_0, 2h_0))}, 
		\norm{\bbb}_{L^\infty(S_\varphi(x_0, 2h_0))},
		h_0, \diam(S_\varphi(x_0, 2h_0)))
		\text{.}
	\end{aligned}
	\end{equation*}

	We now estimate the oscillation of $w$.
	Define
	\begin{equation*}
	\begin{aligned}
		M(t):= \sup_{S_\varphi(x_0,t)} w
		\text{\quad and }
		m(t):=\inf_{S_\varphi(x_0,t)}w
		\text{,}
	\end{aligned}
	\end{equation*}
	and set
	\begin{equation*}
	\begin{aligned}
		w_1(x):=w(x)-m(h)
		\text{\quad and }
		w_2(x):=M(h) - w(x)
		\text{.}
	\end{aligned}
	\end{equation*}
	Then, $w_1$ and $w_2$ are nonnegative solutions to 
	\begin{align*}
		-\dv(\Phi Dw_1 + w_1\bbb)+\bb\cdot Dw_1 &= m(h)\dv\bbb 
		\text{,\quad and }\\
		-\dv(\Phi Dw_2 + w_2\bbb)+\bb\cdot Dw_2 &= -M(h)\dv\bbb 
		\text{}
	\end{align*}
	in $S_\varphi(x_0,h)$.
	Therefore, applying the Harnack inequality in Theorem \ref{thm:harnack} to $w_1$, $w_2$ gives
	\begin{equation}\label{harnhol}
	\begin{aligned}
		M(h/2) - m(h) 
		&\leq C_2 (m(h/2)-m(h)+\norm{m(h)\dv\bbb}_{L^n(S_\varphi(x_0,h_0))}h^{\gamma_2}) 
		\text{,\quad and }\\
		M(h) - m(h/2) 
		&\leq C_2 (M(h)-M(h/2)+\norm{M(h)\dv\bbb}_{L^n(S_\varphi(x_0,h_0))}h^{\gamma_2}) 
		\text{,}
	\end{aligned}
	\end{equation}
	where 
	\begin{equation*}
	\begin{aligned}
		{C}_2 
		&= {C}_2\left(n,\lambda,\Lambda,\ee^*,
		\norm{(\bb,\bbb)}_{L^\infty(S_\varphi(x_0, 2h_0))},
		\norm{\dv\bbb}_{L^n(S_\varphi(x_0, 2h_0))},
		h_0, \diam (S_\varphi(x_0, 2h_0))
		\right)
		\text{,}\\
		\text{and}\\
		{\gamma}_2 &= {\gamma}_2(n, \lambda, \Lambda) > 0
		\text{.}
	\end{aligned}
	\end{equation*}

	Note that $w$ satisfies a nondivergence form equation in $S_\varphi(x_0,h)$. 
	That is,
	\begin{equation*}
	\begin{aligned}
		-\Phi_{ij} D_{ij}w + (\bb-\bbb)\cdot Dw - (\dv\bbb)w = 0
		\text{.}
	\end{aligned}
	\end{equation*}
	As $\dv\bbb\leq 0$, we may apply the maximum principle \cite[Theorem 9.1]{GT}
	using the nondivergence form equation
	to conclude that $w$ takes extreme values on $\partial S$. 
	As $w=u$ on $\partial S$,
	$$|M(h)|, |m(h)| \leq \norm{u}_{L^\infty(S_\varphi(x_0,h))}\text{.}$$ 
	Therefore, as $h\leq h_0$, we have
	\begin{equation*}
	\begin{aligned}
		\norm{m(h)\dv\bbb}_{L^n(S_\varphi(x_0,h_0))}+
		\norm{M(h)\dv\bbb}_{L^n(S_\varphi(x_0,h_0))}
		\leq 2\norm{u}_{L^\infty(S_\varphi(x_0, h_0))} \norm{\dv\bbb}_{L^n(S_\varphi(x_0,2h_0))}
		\text{.}
	\end{aligned}
	\end{equation*}
	Hence, adding the two inequalities in (\ref{harnhol}), we get 
	\begin{equation*}
	\begin{aligned}
		&(1+C_2)(M(h/2)-m(h/2)) \leq \\
		&\qquad (C_2-1)(M(h)-m(h)) + 
		2C_2\norm{\dv\bbb}_{L^n(S_\varphi(x_0, 2h_0))} \norm{u}_{L^\infty(S_\varphi(x_0, h_0))} h^{\gamma_2}
		\text{.}
	\end{aligned}
	\end{equation*}
	Replacing $C_2$ by $C_2 + 2$, we may assume $C_2 > 1$. 
	Setting $\beta : = \frac{C_2-1}{C_2+1}\in (0,1)$ and 
	\begin{equation*}
	\begin{aligned}
		C_3:=\frac{2C_2\norm{\dv\bbb}_{L^n(S_\varphi(x_0, 2h_0))}}{1+C_2}
		\text{,}
	\end{aligned}
	\end{equation*}
	we have
	\begin{equation}\label{6-ide1-2}
	\begin{aligned}
		&\osc(w,S_\varphi(x_0, h/2)) \leq\beta\osc(w,S_\varphi(x_0,h))
		+ C_3
		\norm{u}_{L^\infty(S_\varphi(x_0,h_0))}h^{\gamma_2}
		\text{.}
	\end{aligned}
	\end{equation}
	From the maximum principle, we also have
	\begin{equation}\label{6-ide1-0}
	\begin{aligned}
		\osc(w,S_\varphi(x_0,h))= \osc (w, \partial S_\varphi (x_0, h))
		= \osc (u, \partial S_\varphi (x_0, h))
		\leq \osc(u,S_\varphi(x_0,h))
		\text{.}
	\end{aligned}
	\end{equation}
	
	Recalling $u=v+w$, 
	from (\ref{6-ide1-1}), (\ref{6-ide1-2}), and (\ref{6-ide1-0}) we get
	\begin{equation*}
	\begin{aligned}
		\osc (u, S_\varphi(x_0, h/2))
		&\leq \osc (w, S_\varphi(x_0, h/2))+\osc (v, S_\varphi(x_0, h/2))\\
		&\leq \beta\osc(u, S_\varphi(x_0,h))
		+C_3 \norm{u}_{L^\infty(S_\varphi(x_0,h_0))}h^{\gamma_2}\\
		&\qquad+C_1(\norm{\ff}_{L^\infty(S_\varphi(x_0, 2h_0))}+\norm{f}_{L^r(S_\varphi(x_0, 2h_0))})h^{\gamma_1}
		\text{.}
	\end{aligned}
	\end{equation*}
	Therefore, by a standard argument (see \cite[Lemma 8.23]{GT}), for all $h\leq h_0$ we get
	\begin{equation*}
	\begin{aligned}
		&\osc (u, S_\varphi(x_0, h)) \\
		&\leq C_4\left(\frac{h}{h_0}\right)^{\gamma_3}
		\left(\osc (u, S_\varphi(x_0, h_0))
		+C_3 \norm{u}_{L^\infty(S_\varphi(x_0,h_0))}h_0^{\gamma_2}\right. \\
		&\qquad+\left.C_1(\norm{\ff}_{L^\infty(S_\varphi(x_0, 2h_0))}+\norm{f}_{L^r(S_\varphi(x_0, 2h_0))})h_0^{\gamma_1}\right)\\
		&\leq C_4\left(\frac{h}{h_0}\right)^{\gamma_3}
		\left((2+C_3h_0^{\gamma_2})\norm{u}_{L^\infty(S_\varphi(x_0,h_0))}
		+C_1(\norm{\ff}_{L^\infty(S_\varphi(x_0, 2h_0))}+\norm{f}_{L^r(S_\varphi(x_0, 2h_0))})h_0^{\gamma_1}
		\right)
		\text{,}
	\end{aligned}
	\end{equation*}
	where $C_4 = C_4(\beta)>0$ and $\gamma_3 = \gamma_3(\beta)>0$. 
	This gives the desired oscillation estimate (\ref{oscest}).
	The proof of the Theorem is complete.
\end{proof}

Now, we combine the interior estimate in Lemma \ref{lemma5}
with the H\"older estimate in Corollary \ref{linftyhoelder}
to prove Theorem \ref{mainthm}.
\begin{proof}[Proof of Theorem \ref{mainthm}]
	From Corollary \ref{linftyhoelder},
	for all $x,y\in S_\varphi(x_0, h_0)$, we have
	\begin{equation}\label{8-ide1}
	\begin{aligned}
		|u(x)-u(y)|\leq C_1
		\left(\norm{\ff}_{L^\infty(S_\varphi(x_0, 2h_0))}+\norm{f}_{L^r(S_\varphi(x_0, 2h_0))}
			+\norm{u}_{L^\infty(S_\varphi(x_0,h_0))} \right) |x-y|^\gamma
		\text{,}
	\end{aligned}
	\end{equation}
	where $C_1$ depends on $n$, $\lambda$, $\Lambda$, $r$, $\ee^*$, 
	$\norm{\bb}_{L^\infty(S_\varphi(x_0, 2h_0))}$, $\norm{\bbb}_{L^\infty(S_\varphi(x_0, 2h_0))}$, 
	$\norm{\dv\bbb}_{L^n(S_\varphi(x_0, 2h_0))}$,
	$h_0$, and
	$\mathrm{diam}(S_\varphi(x_0,4h_0))$,
	and $\gamma $ depends on $n$, $\lambda$, $\Lambda$, $\ee^*$,
	$\norm{\bb}_{L^\infty(S_\varphi(x_0, 2h_0))}$, $\norm{\bbb}_{L^\infty(S_\varphi(x_0, 2h_0))}$,
	$\norm{\dv\bbb}_{L^n(S_\varphi(x_0, 2h_0))}$,
	$h_0$, and $\diam(S_\varphi(x_0, 4h_0))$.

	As 
	\begin{equation*}
	\begin{aligned}
		Du^+ = Du \chi_{\{u > 0\}}
		\text{\quad and }
		Du^- = -Du \chi_{\{u<0\}}
		\text{,}
	\end{aligned}
	\end{equation*}
	$u^+$ and $u^-$ are solutions to 
	\begin{equation*}
	\begin{aligned}
		-\dv(\Phi Du^+ + u^+\bbb) + \bb \cdot Du^+ 
		&= f\chi_{\{u>0\}} - \dv(\ff\chi_{\{u>0\}})
		\text{,}\\
		-\dv(\Phi Du^- + u^-\bbb) + \bb \cdot Du^- 
		&= -f\chi_{\{u<0\}} + \dv(\ff\chi_{\{u<0\}})
		\text{.}
	\end{aligned}
	\end{equation*}
	Therefore, we may apply Lemma \ref{lemma5} to $u^+$ and $u^-$ to get
	\begin{equation}\label{8-ide2}
	\begin{aligned}
		\norm{u}_{L^\infty(S_\varphi(x_0, h_0))}
		\leq C_2(h_0^{-\frac{n}{4}}\norm{u}_{L^{2}(S_\varphi(x_0,2h_0))}+
		h_0^{1-\frac{n}{2}}\norm{\ff}_{L^\infty(S_\varphi(x_0,2h_0))}+
		h_0^{1-\frac{n}{2r}}\norm{f}_{L^r(S_\varphi(x_0,2h_0))})
		\text{,}
	\end{aligned}
	\end{equation}
	where 
	\begin{equation*}
	\begin{aligned}
		C_2 = C_2(n,\lambda,\Lambda, r, \ee^*, 
		\norm{\bb}_{L^\infty(S_\varphi(x_0,2h_0))}, 
		\norm{\bbb}_{L^\infty(S_\varphi(x_0,2h_0))}, h_0
		,\diam (S_\varphi (x_0,4h_0)))>0
		\text{.}
	\end{aligned}
	\end{equation*}

	Combining (\ref{8-ide1}) and (\ref{8-ide2}) completes the proof of the Theorem.
\end{proof}

\subsection*{Acknowledgments}
The author would like to thank his advisor, Professor Nam Q. Le,
for suggesting the problem,
and all the advice and guidance the author has received so far.
The author would also like to thank the anonymous referee 
for providing constructive feedback,
which helped the author in improving this paper.

The research of the author was supported in part by NSF grant DMS-2054686.


\begin{thebibliography}{99}

\bibitem{Ab}
Miguel Abreu, \emph{K\"{a}hler geometry of toric varieties and extremal
  metrics}, Internat. J. Math. \textbf{9} (1998), no.~6, 641--651. 

\bibitem{ACDF}
Luigi Ambrosio, Maria Colombo, Guido De~Philippis, and Alessio Figalli,
  \emph{Existence of {E}ulerian solutions to the semigeostrophic equations in
  physical space: the 2-dimensional periodic case}, Comm. Partial Differential
  Equations \textbf{37} (2012), no.~12, 2209--2227. 

\bibitem{B-S}
Peter Bella and Mathias Sch\"{a}ffner, \emph{Local boundedness and {H}arnack
  inequality for solutions of linear nonuniformly elliptic equations}, Comm.
  Pure Appl. Math. \textbf{74} (2021), no.~3, 453--477. 

\bibitem{Cafw2p}
Luis~A. Caffarelli, \emph{Interior {$W^{2,p}$} estimates for solutions of the
  {M}onge-{A}mp\`ere equation}, Ann. of Math. (2) \textbf{131} (1990), no.~1,
  135--150.

\bibitem{CaffarelliC1a}
Luis~A. Caffarelli, \emph{Some regularity properties of solutions of {M}onge
  {A}mp\`ere equation}, Comm. Pure Appl. Math. \textbf{44} (1991), no.~8-9,
  965--969. 

\bibitem{CG97}
Luis~A. Caffarelli and Cristian~E. Guti\'{e}rrez, \emph{Properties of the
  solutions of the linearized {M}onge-{A}mp\`ere equation}, Amer. J. Math.
  \textbf{119} (1997), no.~2, 423--465. 

\bibitem{CR}
Guillaume Carlier and Teresa Radice, \emph{Approximation of variational
  problems with a convexity constraint by {PDE}s of {A}breu type}, Calc. Var.
  Partial Differential Equations \textbf{58} (2019), no.~5, Paper No. 170, 13.

\bibitem{CW}
Albert Chau and Ben Weinkove, \emph{Monge-{A}mp\`ere functionals and the second
  boundary value problem}, Math. Res. Lett. \textbf{22} (2015), no.~4,
  1005--1022. 

\bibitem{CHLS14}
Bohui Chen, Qing Han, An-Min Li, and Li~Sheng, \emph{Interior estimates for the
  {$n$}-dimensional {A}breu's equation}, Adv. Math. \textbf{251} (2014),
  35--46.

\bibitem{CLS}
Bohui Chen, An-Min Li, and Li~Sheng, \emph{The {A}breu equation with
  degenerated boundary conditions}, J. Differential Equations \textbf{252}
  (2012), no.~10, 5235--5259.

\bibitem{DeGiorgi}
Ennio De~Giorgi, \emph{Sulla differenziabilit\`a e l'analiticit\`a delle
  estremali degli integrali multipli regolari}, Mem. Accad. Sci. Torino. Cl.
  Sci. Fis. Mat. Nat. (3) \textbf{3} (1957), 25--43.

\bibitem{DPFS}
G.~De~Philippis, A.~Figalli, and O.~Savin, \emph{A note on interior
  {$W^{2,1+\varepsilon}$} estimates for the {M}onge-{A}mp\`ere equation}, Math.
  Ann. \textbf{357} (2013), no.~1, 11--22. 

\bibitem{Don05}
S.~K. Donaldson, \emph{Interior estimates for solutions of {A}breu's equation},
  Collect. Math. \textbf{56} (2005), no.~2, 103--142. 

\bibitem{FS}
Renjie Feng and G\'{a}bor Sz\'{e}kelyhidi, \emph{Periodic solutions of
  {A}breu's equation}, Math. Res. Lett. \textbf{18} (2011), no.~6, 1271--1279.

\bibitem{Fig}
Alessio Figalli, \emph{Global existence for the semigeostrophic equations via
  {S}obolev estimates for {M}onge-{A}mp\`ere}, Partial differential equations
  and geometric measure theory, Lecture Notes in Math., vol. 2211, Springer,
  Cham, 2018, pp.~1--42.

\bibitem{FSS}
Bruno Franchi, Raul Serapioni, and Francesco Serra~Cassano, \emph{Irregular
  solutions of linear degenerate elliptic equations}, Potential Anal.
  \textbf{9} (1998), no.~3, 201--216. 

\bibitem{GN1}
Cristian~E. Guti\'{e}rrez and Truyen Nguyen, \emph{Interior gradient estimates
  for solutions to the linearized {M}onge-{A}mp\`ere equation}, Adv. Math.
  \textbf{228} (2011), no.~4, 2034--2070.

\bibitem{GN2}
Cristian~E. Guti\'{e}rrez and Truyen Nguyen, 
  \emph{Interior second derivative estimates for solutions to the
  linearized {M}onge-{A}mp\`ere equation}, Trans. Amer. Math. Soc. \textbf{367}
  (2015), no.~7, 4537--4568. 

\bibitem{GT06}
Cristian~E. Guti\'{e}rrez and Federico Tournier, \emph{{$W^{2,p}$}-estimates
  for the linearized {M}onge-{A}mp\`ere equation}, Trans. Amer. Math. Soc.
  \textbf{358} (2006), no.~11, 4843--4872.

\bibitem{GT}
David Gilbarg and Neil~S. Trudinger, \emph{Elliptic partial differential
  equations of second order}, Classics in Mathematics, Springer-Verlag, Berlin,
  2001, Reprint of the 1998 edition. 
  
\bibitem{Han-Lin}
Qing Han and Fanghua Lin, \emph{Elliptic partial differential equations},
  second ed., Courant Lecture Notes in Mathematics, vol.~1, Courant Institute
  of Mathematical Sciences, New York; American Mathematical Society,
  Providence, RI, 2011.

\bibitem{John}
Fritz John, \emph{Extremum problems with inequalities as subsidiary
  conditions}, Studies and {E}ssays {P}resented to {R}. {C}ourant on his 60th
  {B}irthday, {J}anuary 8, 1948, Interscience Publishers, New York, 1948,
  pp.~187--204. 

\bibitem{KLWZ}
Young~Ho Kim, Nam~Q. Le, Ling Wang, and Bin Zhou, \emph{Singular {A}breu
  equations and linearized {M}onge-{A}mp\`ere equations with drifts},
  To appear in J. Eur. Math. Soc. (JEMS),
  \url{https://doi.org/10.4171/jems/1548}
  
\bibitem{LMA_Green}
Nam~Q. Le, \emph{Remarks on the Green's function of the linearized 
	{M}onge-{A}mp\`ere operator}, Manusr. Math.
	\textbf{149} (2016), no.~1, 45--62.

\bibitem{Boundary_Harnack}
Nam~Q. Le, \emph{Boundary {H}arnack inequality for the linearized
  {M}onge-{A}mp\`ere equations and applications}, Trans. Amer. Math. Soc.
  \textbf{369} (2017), no.~9, 6583--6611. 

\bibitem{SG_LMA}
Nam~Q. Le, \emph{H\"{o}lder regularity of the 2{D} dual semigeostrophic equations
  via analysis of linearized {M}onge-{A}mp\`ere equations}, Comm. Math. Phys.
  \textbf{360} (2018), no.~1, 271--305. 

\bibitem{Le18}
Nam~Q. Le, \emph{On the {H}arnack inequality for degenerate and singular elliptic
  equations with unbounded lower order terms via sliding paraboloids}, Commun.
  Contemp. Math. \textbf{20} (2018), no.~1, 1750012, 38. 
  
\bibitem{Singular_Abreu}
Nam~Q. Le, \emph{Singular {A}breu equations and minimizers of convex functionals
  with a convexity constraint}, Comm. Pure Appl. Math. \textbf{73} (2020),
  no.~10, 2248--2283. 

\bibitem{Convex_Approx}
Nam~Q. Le, \emph{On approximating minimizers of convex functionals with a
  convexity constraint by singular {A}breu equations without uniform
  convexity}, Proc. Roy. Soc. Edinburgh Sect. A \textbf{151} (2021), no.~1,
  356--376. 

\bibitem{Twisted_Harnack}
Nam~Q. Le, \emph{Twisted {H}arnack inequality and approximation of variational
  problems with a convexity constraint by singular {A}breu equations}, Adv.
  Math. \textbf{434} (2023). 

\bibitem{Le24}
Nam~Q. Le, \emph{Analysis of {M}onge-{A}mpère equations}, Graduate Studies in
  Mathematics, vol. 240, American Mathematical Society, 2024.

\bibitem{LNW2p}
Nam~Q. Le and Truyen Nguyen, \emph{Global {$W^{2,p}$} estimates for solutions
  to the linearized {M}onge-{A}mp\`ere equations}, Math. Ann. \textbf{358}
  (2014), no.~3-4, 629--700. 

\bibitem{LN}
Nam~Q. Le and Truyen Nguyen, \emph{Global {$W^{1,p}$} estimates for solutions to the linearized
  {M}onge-{A}mp\`ere equations}, J. Geom. Anal. \textbf{27} (2017), no.~3,
  1751--1788. 

\bibitem{LS}
Nam~Q. Le and Ovidiu Savin, \emph{Boundary regularity for solutions to the linearized
  {M}onge-{A}mp\`ere equations}, {Arch. Ration. Mech. Anal.}
  \textbf{210} (2013), no.~3, 813--836.
  
\bibitem{Abreu_HD}
Nam~Q. Le and Bin Zhou, \emph{Solvability of a class of singular fourth order
  equations of {M}onge-{A}mp\`ere type}, Ann. PDE \textbf{7} (2021), no.~2,
  Paper No. 13, 32. 

\bibitem{Loeper05}
G.~Loeper, \emph{On the regularity of the polar factorization for time
  dependent maps}, Calc. Var. Partial Differential Equations \textbf{22}
  (2005), no.~3, 343--374. 

\bibitem{Maldonado}
Diego Maldonado, \emph{Harnack's inequality for solutions to the linearized
  {M}onge-{A}mp\`ere operator with lower-order terms}, J. Differential
  Equations \textbf{256} (2014), no.~6, 1987--2022. 

\bibitem{MalW2}
Diego Maldonado, \emph{On the {$W^{2,1+\varepsilon}$}-estimates for the
  {M}onge-{A}mp\`ere equation and related real analysis}, Calc. Var. Partial
  Differential Equations \textbf{50} (2014), no.~1-2, 93--114. 

\bibitem{M17}
Diego Maldonado, \emph{{$W^{1,p}_\varphi$}-estimates for {G}reen's functions of
  the linearized {M}onge-{A}mp\`ere operator}, Manuscripta Math. \textbf{152}
  (2017), no.~3-4, 539--554.

\bibitem{M18}
Diego Maldonado, \emph{On certain degenerate and singular elliptic {PDE}s {I}:
  {N}ondivergence form operators with unbounded drifts and applications to
  subelliptic equations}, J. Differential Equations \textbf{264} (2018), no.~2,
  624--678.

\bibitem{M19}
Diego Maldonado, \emph{On certain degenerate and singular elliptic {PDE}s {II}:
  {D}ivergence-form operators, {H}arnack inequalities, and applications}, J.
  Differential Equations \textbf{266} (2019), no.~6, 3679--3731.

\bibitem{Moser}
J\"{u}rgen Moser, \emph{On {H}arnack's theorem for elliptic differential
  equations}, Comm. Pure Appl. Math. \textbf{14} (1961), 577--591. 

\bibitem{MS}
M.~K.~V. Murthy and G.~Stampacchia, \emph{Boundary value problems for some
  degenerate-elliptic operators}, Ann. Mat. Pura Appl. (4) \textbf{80} (1968),
  1--122. 

\bibitem{Nash}
J.~Nash, \emph{Continuity of solutions of parabolic and elliptic equations},
  Amer. J. Math. \textbf{80} (1958), 931--954. 

\bibitem{S}
Ovidiu Savin, \emph{A {L}iouville theorem for solutions to the linearized
  {M}onge-{A}mp\`ere equation}, {Discrete Contin. Dyn. Syst.}
  \textbf{28} (2010), no.~3, 865--873.

\bibitem{Schmidt}
Thomas Schmidt, \emph{{${\rm W}^{2,1+\varepsilon}$} estimates for the
  {M}onge-{A}mp\`ere equation}, Adv. Math. \textbf{240} (2013), 672--689.

\bibitem{TZ}
Lin Tang and Qian Zhang, \emph{Global {$W^{2,p}$} regularity on the linearized
  {M}onge-{A}mp\`ere equation with {VMO} type coefficients}, Results Math.
  \textbf{77} (2022), no.~2, Paper No. 94, 88.

\bibitem{masobolev}
Gu-Ji Tian and Xu-Jia Wang, \emph{A class of {S}obolev type inequalities},
  Methods Appl. Anal. \textbf{15} (2008), no.~2, 263--276.

\bibitem{Trudinger1973}
Neil~S. Trudinger, \emph{Linear elliptic operators with measurable
  coefficients}, Ann. Scuola Norm. Sup. Pisa Cl. Sci. (3) \textbf{27} (1973),
  265--308. 

\bibitem{TW00}
Neil~S. Trudinger and Xu-Jia Wang, \emph{The {B}ernstein problem for affine
  maximal hypersurfaces}, Invent. Math. \textbf{140} (2000), no.~2, 399--422.

\bibitem{TW05}
Neil~S. Trudinger and Xu-Jia Wang, \emph{The affine {P}lateau problem}, J.
  Amer. Math. Soc. \textbf{18} (2005), no.~2, 253--289.

\bibitem{TW08}
Neil~S. Trudinger and Xu-Jia Wang, \emph{Boundary regularity for the
  {M}onge-{A}mp\`ere and affine maximal surface equations}, Ann. of Math. (2)
  \textbf{167} (2008), no.~3, 993--1028. 
  
\bibitem{W}
Ling Wang, \emph{Interior {H}\"older regularity of the linearized 
	{M}onge-{A}mp\`ere equation}, 
	Calc. Var. Partial Differential Equations \textbf{64} (2025),
	no.~1, Paper No. 17, 16.
	
\bibitem{WZ}
Ling Wang and Bin Zhou,
	\emph{Interior estimates for {M}onge-{A}mp\`ere type fourth order equations}.
	Rev. Mat. Iberoam. \textbf{39} (2023), no.~5, 1895--1923.

\bibitem{Zhou}
Bin Zhou, \emph{The first boundary value problem for {A}breu's equation}, Int.
  Math. Res. Not. IMRN (2012), no.~7, 1439--1484.

\end{thebibliography}
\end{document}